\documentclass[a4paper,11pt]{svjour3}

\usepackage{ifthen}

\newboolean{AppendixOutsourcing}
\setboolean{AppendixOutsourcing}{false}

\usepackage[ngerman,english]{babel}
\usepackage[T1]{fontenc}
\usepackage[latin1]{inputenc}
\usepackage[numbers,comma,sort&compress]{natbib}
\usepackage{verbatim} 
\usepackage{multirow}
\usepackage{hyperref}

\usepackage{enumitem}

\newif\ifpdf
\ifx\pdfoutput\undefined
\usepackage{graphicx}
\pdffalse
\else
\pdftrue
\usepackage{graphicx}
\usepackage{epstopdf}
\fi

\usepackage{fancyhdr}

\usepackage{amsmath,amsfonts,amssymb,amscd,amsthm,xspace}
\theoremstyle{plain}
\newtheorem{algorithm}[theorem]{Algorithm}     
\newtheorem{notation}[theorem]{Notation}     
\usepackage[centerlast,small,sc]{caption}

\usepackage[numbers,comma,sort&compress]{natbib}
\usepackage{url}
\newcommand{\komma}{\textnormal{~,}}
\newcommand{\punkt}{\textnormal{~.}}

\newcommand{\sgn}{\mathrm{sgn}}

\newcommand{\var}{\mathrm{var}}

\newcommand{\opt}{\mathrm{opt}}

\newcommand{\rp}{\textnormal{rp}}
\newcommand{\NforMPBNGC}{\textnormal{Nb}}
\newcommand{\NforSolvOpt}{\textnormal{Nc}}
\newcommand{\TotalTime}{\textnormal{$\textnormal{t}_{1}$}}
\newcommand{\PartTime}{\textnormal{$\textnormal{t}_{2}$}}

\renewcommand{\forall}{\textnormal{for all }}

\newcommand{\Nfunc}{Y}
\newcommand{\NdimCer}{N}

\newcommand{\Ffrak}{\mathfrak F}  
\newcommand{\wVariable}{w}        
\newcommand{\wDimension}{q}       

\newcommand{\zeroVector}[1]{0}

\newcommand{\zeroMatrix}[1]{0}

\newcommand{\idMatrix}[1]{I}

\newcommand{\zeroVectorI}[1]{0}

\newcommand{\zeroMatrixI}[1]{0}

\newcommand{\Sym}[1]{\mathbb{R}_{\mathrm{sym}}^{#1\times#1}}

\newcommand{\Tril}[1]{\mathbb{R}_{\mathrm{tril}}^{#1\times#1}}

\newcommand{\Triu}[1]{\mathbb{R}_{\mathrm{triu}}^{#1\times#1}}

\newcommand{\sTriu}[1]{\mathbb{R}_{\mathrm{striu}}^{#1\times#1}}

\newcommand{\Rneg}{\mathbb{R}_{\leq0}}

\newcommand{\Rpos}{\mathbb{R}_{\geq0}}

\usepackage{color} 
\definecolor{gold}{rgb}{0.85,.66,0}

\definecolor{orange}{rgb}{1,0.5,0}   
\definecolor{violet}{rgb}{0.5,0,0.5} 

\newcommand{\GenaueAngabeANone}{}
\newcommand{\GenaueAngabeANtwo}{}
\newcommand{\GenaueAngabeANthree}{}
\newcommand{\GenaueAngabeHD}{}

\newcommand{\GenaueAngabeConvergenceTheory}{p.~7,~3.1~Theoretical basics}  

\newcommand{\GenaueAngabeOne}{p.~147~ff,~Chapter~5}                        

\newcommand{\GenaueAngabeThree}{}
\newcommand{\GenaueAngabeFour}{}
\newcommand{\GenaueAngabeFive}{}
\newcommand{\GenaueAngabeSix}{}

\newcommand{\refh}[1]{\textnormal{(\ref{#1})}}   

\def\fullTitle{{Certificates of infeasibility via nonsmooth optimization}}

\def\AuthorOne{{Hannes Fendl}}
\def\AuthorTwo{{Arnold Neumaier}}
\def\AuthorThree{{Hermann Schichl}}

\ifpdf
\hypersetup{
    pdftitle  = {\fullTitle},
    pdfauthor = {\AuthorOne, \AuthorTwo, and \AuthorThree},
}
\fi

\usepackage[color]{changebar}
\cbcolor{red}
\nochangebars

\newcommand{\cbstartRED}{}
\newcommand{\cbendRED}{}
\newcommand{\cbmathRED}{}

\newlength{\widebarargwidth}
\newlength{\widebarwidth}
\newlength{\widebarargheight}
\newlength{\widebarargdepth}

\def\HauptAlgorithmus{{the second order bundle algorithm }}
\def\HauptAlgorithmusKomma{{the second order bundle algorithm, }}

\newcommand{\notocsection}[1]{\refstepcounter{section}\section*{\thesection \quad #1}}

\newcommand{\emptyh}[1]{}

\begin{document}

\title{\fullTitle}
\author{\AuthorOne\thanks{This research was supported by the Austrian Science Found (FWF) Grant Nr.~P22239-N13.} \and \AuthorTwo \and \AuthorThree}
\institute{Faculty of Mathematics, University of Vienna, Austria\\
  Oskar-Morgenstern-Pl.~1, A-1090 Wien, Austria\\
\email{hermann.schichl@univie.ac.at,arnold.neumaier@univie.ac.at}}


\maketitle

\begin{abstract}
An important aspect in the solution process of constraint satisfaction problems is to identify exclusion boxes which are boxes that do not contain feasible points. This paper presents a certificate of infeasibility for finding such boxes by solving a linearly constrained nonsmooth optimization problem. Furthermore, the constructed certificate can be used to enlarge an exclusion box by solving a nonlinearly constrained nonsmooth optimization problem. 
\end{abstract}



\section{Introduction}
An important area of modern research is global optimization as it occurs very frequently in applications (extensive surveys on global optimization can be found in \citet{neumaier2004complete}, \citet{Flo,Flo.MI}, \citet{Han}, and \citet{Kea}). A method for solving global optimization problems efficiently is by using a branch and bound algorithm (as, e.g., BARON by \citet{Sah,Sah.baron}, the COCONUT environment by \citet{schichl2004global,schichlcoconut,SchichlHabilitation}, or LINGO by \citet{schichl:LINGO}), which divides the feasible set into smaller regions and then tries to exclude regions that cannot contain a global optimizer. Therefore, it is important to have tools which allow to identify such regions. In this paper we will present a method which is able to find such regions for a CSP (constraint satisfaction problem), i.e.~for a global optimization problem with a constant objective function,
\cbstartRED by generalizing the approach from \citet{HannesDissertation}.\cbendRED
~\\
\cbstartRED\textbf{Certificate of infeasibility}. For this purpose we consider the CSP
\begin{equation}
\begin{split}
F(x)&\in\boldsymbol{F}\\
  x &\in\boldsymbol{x}
\label{coi:CSP}
\end{split}
\end{equation}
with $F:\mathbb{R}^n\longrightarrow\mathbb{R}^m$, $\boldsymbol{x}\in\mathbb{IR}^n$, $\boldsymbol{F}\in\mathbb{IR}^m$,
\cbendRED
and we assume that a solver, which is able to solve a CSP, takes the box $\boldsymbol{u}:=[\underline{u},\overline{u}]\subseteq\boldsymbol{x}$ into consideration during the solution process. We constructed a certificate of infeasibility $f$, which is a nondifferentiable and nonconvex function in general, with the following property: If there exists a vector $y$ with
\begin{equation}
f(y,\underline{u},\overline{u})<0\komma
\label{Abstract:CertificateNegative}
\end{equation}
then the CSP (\ref{coi:CSP}) has no feasible point in $\boldsymbol{u}$ and consequently this box can be excluded for the rest of the solution process. Therefore,
a box $\boldsymbol{u}$ for which (\ref{Abstract:CertificateNegative}) holds is called an \textbf{exclusion box}.

Easy examples immediately show that there exist CSPs which have boxes that satisfy $(\ref{Abstract:CertificateNegative})$, so it is worth to pursue this approach further.
~\\
\textbf{Exclusion boxes}. The obvious way for finding an exclusion box for the CSP (\ref{coi:CSP}) is to minimize $f$
\begin{equation}
\begin{split}
\min_{y}{f(y,\underline{u},\overline{u})}
\label{Abstract:coi3:Optimierungsproblem}
\end{split}
\end{equation}
and stop the minimization if a negative function value occurs. Since
modern solvers offer
many other possibilities for treating a box, we do not want to spend too much time for this minimization problem. Therefore, the idea is to let a nonsmooth solver only perform a few steps for solving (\ref{Abstract:coi3:Optimierungsproblem}).

To find at least an exclusion box $\boldsymbol{v}:=[\underline{v},\overline{v}]\subseteq\boldsymbol{u}$ with $\underline{v}+r\leq\overline{v}$, where $r\in(\zeroVector{n},\overline{u}-\underline{u})$ is fixed, we can try to solve the linearly constrained problem
\begin{equation*}
\begin{split}
&\min_{y,\underline{v},\overline{v}}{f(y,\underline{v},\overline{v})}\\
&\textnormal{ s.t. }[\underline{v}+r,\overline{v}]\subseteq\boldsymbol{u}\punkt
\end{split}
\end{equation*}
Another important aspect in this context is to enlarge an exclusion box $\boldsymbol{v}$ by solving
\begin{equation}
\begin{split}
&\max_{y,\underline{v},\overline{v}}{\mu(\underline{v},\overline{v})}\\
&\textnormal{ s.t. }f(y,\underline{v},\overline{v})\leq\delta\komma~[\underline{v},\overline{v}]\subseteq\boldsymbol{u}\komma
\label{Abstract:coi3:OptimierungsproblemNB}
\end{split}
\end{equation}
where $\delta<0$ is given and $\mu$ measures the magnitude of the box $\boldsymbol{v}$ (e.g., $\mu(\underline{v},\overline{v}):=\lvert\overline{v}-\underline{v}\rvert_{_1}$). Since only feasible points of (\ref{Abstract:coi3:OptimierungsproblemNB}) are useful for enlarging an exclusion box and we only want to perform a few steps of a nonsmooth solver as before, we expect benefits from a nonsmooth solver that only creates feasible iterates because then the current best point can always be used for our purpose. For proofs in explicit detail we refer the reader to \citet[\GenaueAngabeOne]{HannesDissertation}.

The paper is organized as follows: In Section \ref{Paper:PresentationOfTheApplication} we first recall the basic facts of interval analysis which are necessary for introducing the certificate of infeasibility which is done afterwards. Then we discuss some important properties of the certificate and we explain in detail how the certificate is used for obtaining exclusion boxes in a CSP by applying a nonsmooth solver.
In Section \ref{Paper:ImplementationIssuesForFindingExclusionBoxes} we explain how we obtain a starting point for optimization problems to which we apply the nonsmooth solver.

Throughout the paper we use the following notation: We denote the non-negative real numbers by $\Rpos:=\lbrace x\in\mathbb{R}:~x\geq0\rbrace$ (and analogously for $\leq$ as well as $>$). Furthermore, we denote the $p$-norm of $x\in\mathbb{R}^n$ by $\lvert x\rvert_{_p}$ for $p\in\lbrace1,2,\infty\rbrace$.

\section{Presentation of the application}
\label{Paper:PresentationOfTheApplication}
After summarizing the most basic facts of interval analysis, we construct the certificate of infeasibility in this section. Furthermore, we discuss how a nonsmooth solver can use this certificate to obtain an exclusion box in a CSP.
\subsection{Interval arithmetic}
We recall some basic facts on interval arithmetic from, e.g., \citet{NeumaierIV}.
We denote a box (also called interval vector) by $\boldsymbol{x}=[\underline{x},\overline{x}]$ and the set of all boxes by $\mathbb{IR}^n:=\lbrace\boldsymbol{x}:\boldsymbol{x}=[\underline{x},\overline{x}],\underline{x}\leq\overline{x}\rbrace$. For $S\subseteq\mathbb{R}$ bounded, we denote the hull of $S$ by $\square{S}:=[\inf{S},\sup{S}]$. We extend the arithmetic operations and functions $\varphi:\mathbb{R}\longrightarrow\mathbb{R}$ to boxes by defining
\begin{equation}
\varphi(\boldsymbol{x})
:=
\square\lbrace\varphi(x):x\in\boldsymbol{x}\rbrace
\label{coi:DefIntervallFkt}
\punkt
\end{equation}
For every expression $\Phi$ of $\varphi:\mathbb{R}^n\longrightarrow\mathbb{R}$ which is a composition of arithmetic operations and elementary functions the fundamental theorem of interval arithmetic holds
\begin{equation}
[\inf_{x\in\boldsymbol{x}}\varphi(x),\sup_{x\in\boldsymbol{x}}\varphi(x)]\subseteq\Phi(\boldsymbol{x})
\punkt
\label{coi:Satz:Hauptsatz}
\end{equation}

\subsection{Certificate of infeasibility}
\cbstartRED
Let $s:\mathbb{R}^m\times\mathbb{R}^n\times\mathbb{R}^{\wDimension}\times\mathbb{IR}^n\longrightarrow\mathbb{IR}$ be a function. We assume that $Z:\mathbb{R}^m\times\mathbb{R}^n\times\mathbb{R}^{\wDimension}\times\mathbb{R}^n\times\mathbb{R}^n\longrightarrow\mathbb{R}$
\begin{equation}
Z(y,z,\wVariable,\underline{x},\overline{x})
:=
\sup{s(y,z,\wVariable,\boldsymbol{x})}
\label{coi3:Def:Z}
\komma
\end{equation}
where $\boldsymbol{x}=[\underline{x},\overline{x}]\in\mathbb{IR}^n$, satisfies
\begin{equation}
Z(y,z,\wVariable,\underline{x},\overline{x})
\geq
\sup_{x\in\boldsymbol{x}}{y^T\big(F(x)-F(z)\big)}
\punkt
\label{Def:phi:InclusionProperty}
\end{equation}
\begin{example}
If we set $\wVariable=(R,S)\in\Triu{n}\times\sTriu{n}$, where we denote the linear space of the upper resp.~strictly upper triangular $n\times n$-matrices by $\Triu{n}\cong\mathbb{R}^{n_1}$ with $n_1:=\tfrac{1}{2}n(n+1)$ resp.~$\sTriu{n}\cong\mathbb{R}^{n_0}$ with $n_0:=\tfrac{1}{2}(n-1)n$, which implies $\wDimension=n_0+n_1=n^2$, and if we define
\begin{equation}
s_1(y,z,R,S,\boldsymbol{x})
:=
\Big(\sum_{k=1}^m{y_k\Ffrak_k[z,\boldsymbol{x}]+(\boldsymbol{x}-z)^T(R^TR+S^T-S)}\Big)(\boldsymbol{x}-z)
\label{example:Def:sk2}
\end{equation}
then the corresponding $Z$ satisfies (\ref{Def:phi:InclusionProperty}) because:
Due to the skew-symmetry of $S^T-S$, we have
\begin{equation}
y^T\big(F(x)-F(z)\big)
+(x-z)^T(R^TR+S^T-S)(x-z)
\geq
y^T\big(F(x)-F(z)\big)
\punkt
\label{example:Estimation:sk2:Proof:Additional1}
\end{equation}
for all $x\in\boldsymbol{x}$.
Since the slope expansion
\begin{equation}
F_k(x)
=
F_k(z)+F_k[z,x](x-z)
\label{example:Estimation:sk2}
\komma
\end{equation}
where the slope $F_k[z,x]\in\mathbb{R}^{1\times n}$, holds for all $x,z\in\mathbb{R}^n$
(cf., e.g., \citet{NeumaierIV}),
we obtain for all $x\in\boldsymbol{x}$
\begin{equation}
y^T\big(F(x)-F(z)\big)
+(x-z)^T(R^TR+S^T-S)(x-z)
\subseteq
s_1(y,z,R,S,\boldsymbol{x})
\label{example:Estimation:sk2:Proof:Final}
\end{equation}
due to (\ref{example:Estimation:sk2}), (\ref{coi:DefIntervallFkt}), (\ref{coi:Satz:Hauptsatz}), and (\ref{example:Def:sk2}).
Now we obtain (\ref{Def:phi:InclusionProperty}) due to (\ref{example:Estimation:sk2:Proof:Additional1}), (\ref{example:Estimation:sk2:Proof:Final}), and (\ref{coi3:Def:Z}).
\end{example}
\begin{proposition}
It holds for all $z\in\boldsymbol{x}$
\begin{equation}
Z(y,z,w,\underline{x},\overline{x})\geq0
\punkt
\label{coi3:SatzZnichtnegativ}
\end{equation}
\end{proposition}
\begin{proof}
(\ref{coi3:SatzZnichtnegativ}) follows from (\ref{Def:phi:InclusionProperty}) and the assumption that $z\in\boldsymbol{x}$.
\end{proof}
\begin{definition}
\label{AN:Definition:NfuncZf}
We define $\Nfunc:\mathbb{R}^m\times\mathbb{R}^n\longrightarrow\mathbb{R}$ and $f:\mathbb{R}^m\times\mathbb{R}^n\times\mathbb{R}^{\wDimension}\times\mathbb{R}^n\times\mathbb{R}^n\longrightarrow\mathbb{R}$ by
\begin{align}
\Nfunc(y,z)
&:=
\inf{y^T\big(\boldsymbol{F}-F(z)\big)}
\label{coi3:Def:N}
\\
f(y,z,w,\underline{x},\overline{x})
&:=
\frac{Z(y,z,w,\underline{x},\overline{x})-\max{\big(0,\Nfunc(y,z)\big)}}{T(y,z,w,\underline{x},\overline{x})}
\label{coi3:Def:f}
\komma
\end{align}
where $T:\mathbb{R}^m\times\mathbb{R}^n\times\mathbb{R}^{\wDimension}\times\mathbb{R}^n\times\mathbb{R}^n\longrightarrow\mathbb{R}_{>0}$ is positive, continuous and differentiable almost everywhere.
\end{definition}
\begin{remark}
\label{COI:Remark:CertificateNumberOfVariablesAndNonsmooth}
$f$ from (\ref{coi3:Def:f}) depends on $\NdimCer=m+3n+\wDimension$ variables and is not differentiable everywhere and not convex (in general).
\end{remark}
Now we state the main theorem for our application.
\begin{theorem}
\label{coi3:Satz:UnzulässigkeitszertifikatVariableIntervallgrenze}
If there exist $y\in\mathbb{R}^m$, $\underline{x}\leq z\leq\overline{x}\in\mathbb{R}^n$, and $\wVariable\in\mathbb{R}^{\wDimension}$ with $f(y,z,w,\underline{x},\overline{x})<0$, then for all $x\in\boldsymbol{x}$ there exists $k\in\lbrace1,\dots,m\rbrace$ with $F_k(x)\not\in\boldsymbol{F}_k$, i.e.~there is no $x\in\boldsymbol{x}$ with $F(x)\in\boldsymbol{F}$, i.e.~there is no feasible point.
\end{theorem}
\begin{proof}
(by contradiction) Suppose that there exists $\hat{x}\in\boldsymbol{x}:=[\underline{x},\overline{x}]$ with $F(\hat{x})\in\boldsymbol{F}$. By assumption there exist $y\in\mathbb{R}^m$ and $z\in\boldsymbol{x}\subseteq\mathbb{IR}^n$ with $f(y,z,\wVariable,\underline{x},\overline{x})<0$, which is equivalent to 
$
Z(y,z,\wVariable,\underline{x},\overline{x})<\Nfunc(y,z)
$
due to (\ref{coi3:Def:f}) and (\ref{coi3:SatzZnichtnegativ}). Since
\begin{equation*}
Z(y,z,\wVariable,\underline{x},\overline{x})
\geq
\sup_{x\in\boldsymbol{x}}{y^T\big(F(x)-F(z)\big)}
\geq
y^T\big(F(x)-F(z)\big)
\end{equation*}
for all $x\in\boldsymbol{x}$ due to (\ref{Def:phi:InclusionProperty})
and
\begin{equation*}
\Nfunc(y,z)
\leq
\inf_{\tilde{F}\in\boldsymbol{F}}{y^T\big(\tilde{F}-F(z)\big)}
\leq
y^T\big(\tilde{F}-F(z)\big)
\end{equation*}
for all $\tilde{F}\in\boldsymbol{F}$ due to (\ref{coi3:Def:N}) and (\ref{coi:Satz:Hauptsatz}),
we obtain for all $x\in\boldsymbol{x}$ and for all $\tilde{F}\in\boldsymbol{F}$
\begin{equation*}
y^T\big(F(x)-F(z)\big)
\leq
Z(y,z,\wVariable,\underline{x},\overline{x})
<
\Nfunc(y,z)
\leq
y^T\big(\tilde{F}-F(z)\big)
\komma
\end{equation*}
which implies that
we have
$
y^TF(x)
<
y^T\tilde{F}
$
for all $x\in\boldsymbol{x}$ and for all $\tilde{F}\in\boldsymbol{F}$.
Now, choosing $x=\hat{x}\in\boldsymbol{x}$ and $\tilde{F}=F(\hat{x})\in\boldsymbol{F}$ in
the last inequality
yields a contradiction.
\qedhere
\end{proof}
\cbendRED
The following proposition gives in particular a hint how the $y$-component of a starting point should be chosen (cf.~(\ref{coi:yStartpunktFLunboundedCOMPACT})).
\begin{proposition}
We have for all $y\in\mathbb{R}^m$ and $z\in\mathbb{R}^n$
\begin{equation}
\Nfunc(y,z)
=
\sum_{k=1}^m
\left\lbrace
\begin{array}{ll}
y_k\big(\underline{F}_k-F_k(z)\big) & \textnormal{for }y_k\geq0\\
y_k\big(\overline{F}_k-F_k(z)\big)  & \textnormal{for }y_k<0   \punkt
\end{array}
\right.
\label{coi:SatzNexplizit}
\end{equation}
Furthermore,
let
$I,J\subseteq\lbrace1,\dots,m\rbrace$ satisfy
$I\not=\emptyset~\vee~J\not=\emptyset$,
$\underline{F}_i=-\infty~\wedge~y_i>0$ for all $i\in I$ and
$\overline{F}_j=\infty~\wedge~y_j<0$ for all $j\in J$,
then we have for all $z\in\mathbb{R}^n$
\begin{equation}
\Nfunc(y,z)=-\infty\punkt
\label{coi:Proposition:NyzEqualsMinusInfty}
\end{equation}
\end{proposition}
\begin{proof}
(\ref{coi:SatzNexplizit}) holds because of (\ref{coi3:Def:N}). For obtaining (\ref{coi:Proposition:NyzEqualsMinusInfty}), consider without loss of generality $\underline{F}_1=-\infty~\wedge~y_1>0$ and $\overline{F}_2=\infty~\wedge~y_2<0$, then the desired result follows from (\ref{coi:SatzNexplizit}).
\qedhere
\end{proof}

\subsection{Properties of the certificate for quadratic F}
\cbstartRED
In this subsection we consider the special case of $f$ with quadratic $F$, i.e.
\begin{equation}
F_k(x)
=
c_k^Tx+x^TC_kx
\label{coi:DefFk}
\end{equation}
with $c_k\in\mathbb{R}^n$ and $C_k\in\mathbb{R}^{n\times n}$ for $k=1,\dots,m$. Since
\begin{equation*}
F_k(x)-F_k(z)
=
\big(c_k^T+x^TC_k+z^TC_k^T\big)(x-z)
\end{equation*}
for all $x,z\in\mathbb{R}^n$, the slope expansion from (\ref{example:Estimation:sk2})
holds with
\begin{equation}
F_k[z,x]
=
c_k^T+x^TC_k+z^TC_k^T
\punkt
\label{coi:Def:QuadraticFk:Slope}
\end{equation}
\begin{proposition}
\label{coi:corollary:OriginalCertificate}
Let $F_k$ be quadratic and set
\begin{equation}
  \begin{aligned}
C(y)
:=
\sum_{k=1}^mC_ky_k
,~
c(y,z)
&:=
\sum_{k=1}^mc_ky_k+\big(C(y)+C(y)^T\big)z
,~
A(y,R,S)\\
&:=
C(y)+R^TR+S^T-S
\punkt
\label{coi:DefCCOMPACT}
  \end{aligned}
\end{equation}
Then \refh{Def:phi:InclusionProperty} is satisfied for
\begin{equation}
s_2(y,z,R,S,\boldsymbol{x})
:=
\big(c(y,z)^T+(\boldsymbol{x}-z)^TA(y,R,S)\big)(\boldsymbol{x}-z)
\punkt
\label{coi:DefZ}
\end{equation}
\end{proposition}
\begin{proof}
Since
$
\sum_{k=1}^m{y_kF_k[z,x]}
=
c(y,z)^T+(x-z)^TC(y)
$
due to
(\ref{coi:Def:QuadraticFk:Slope}) and
(\ref{coi:DefCCOMPACT}),
we obtain
$
\sum_{k=1}^m{y_kF_k[z,x]}+(x-z)^T(R^TR+S^T-S)
=
c(y,z)^T+(x-z)^TA(y,R,S)
$
due to (\ref{coi:DefCCOMPACT}),
which implies that $s_2$ from (\ref{coi:DefZ}) has the same structure as $s_1$ from (\ref{example:Def:sk2}), and consequently we obtain that (\ref{Def:phi:InclusionProperty}) holds for $s_2$, too.
\cbendRED
\end{proof}
\begin{proposition}
Let $F_k$ be quadratic, then we have for all $p\in\mathbb{R}^m$ and $\alpha\in\mathbb{R}$
\begin{equation}
C(y+\alpha p)
=
C(y)+\alpha C(p)
\komma\quad
c(y+\alpha p,z)
=
c(y,z)+\alpha c(p,z)
\label{coi:SatzCCOMPACT}
\end{equation}
and furthermore we have for all $\kappa\geq0$
\begin{equation}
A(\kappa^2y,\kappa R,\kappa^2S)
=
\kappa^2A(y,R,S)
\punkt
\label{coi:Proposition:Ascaling}
\end{equation}
\end{proposition}
\begin{proof}
(\ref{coi:SatzCCOMPACT}) holds due to (\ref{coi:DefCCOMPACT}). (\ref{coi:Proposition:Ascaling}) holds due to (\ref{coi:DefCCOMPACT}) and (\ref{coi:SatzCCOMPACT}).
\qedhere
\end{proof}
\begin{proposition}
\label{coi:Proposition:ReduceScalabilityDependency}
Let $F_k$ be quadratic. If the positive function $T:\mathbb{R}^m\times\mathbb{R}^n\times\mathbb{R}^{n_1}\times\mathbb{R}^{n_0}\times\mathbb{R}^n\times\mathbb{R}^n\longrightarrow\mathbb{R}_{>0}$ satisfies the (partial) homogeneity condition
\begin{equation}
T(\kappa^2y,z,\kappa R,\kappa^2S,\underline{x},\overline{x})=\kappa^2T(y,z,R,S,\underline{x},\overline{x})
\label{coi:THomogen}
\end{equation}
for all $\kappa>0$, $y\in\mathbb{R}^m$, $z\in[\underline{x},\overline{x}]\in\mathbb{IR}^n$, $R\in\Triu{n}$ and $S\in\sTriu{n}$, then
the certificate $f$ from \refh{coi3:Def:f} is (partially) homogeneous
\begin{equation}
f(\kappa^2y,z,\kappa R,\kappa^2S,\underline{x},\overline{x})=f(y,z,R,S,\underline{x},\overline{x})
\punkt
\label{coi:Proposition:ReduceScalabilityDependency:Equation}
\end{equation}
\end{proposition}
\begin{proof}
Since $F_k$ is quadratic by assumption, the statements of Proposition \ref{coi:corollary:OriginalCertificate} hold. By using (\ref{coi:SatzCCOMPACT}) and (\ref{coi:Proposition:Ascaling}) we calculate
\begin{equation}
c(\kappa^2y,z)^T+(\boldsymbol{x}-z)^TA(\kappa^2y,\kappa R,\kappa^2S)
=
\kappa^2\big(c(y,z)^T+(\boldsymbol{x}-z)^TA(y,R,S)\big)
\label{coi:THomogen:Proof2}
\end{equation}
and therefore $Z(\kappa^2y,z,\kappa R,\kappa^2S,\underline{x},\overline{x})=\kappa^2Z(y,z,R,S,\underline{x},\overline{x})$ follows due to
\cbstartRED
(\ref{coi3:Def:Z}),
\cbendRED
(\ref{coi:DefZ}),
and (\ref{coi:THomogen:Proof2}). Furthermore, we have $\Nfunc(\kappa^2y,z)=\kappa^2\Nfunc(y,z)$ due to (\ref{coi3:Def:N}) and, hence, we obtain $\max{\big(0,\Nfunc(\kappa^2y,z)\big)}=\kappa^2\max{\big(0,\Nfunc(y,z)\big)}$. Consequently, (\ref{coi3:Def:f}) and (\ref{coi:THomogen}) imply (\ref{coi:Proposition:ReduceScalabilityDependency:Equation}).
\qedhere
\end{proof}
\begin{remark}
\label{coi:Remark:ReduceScalabilityDependency}
The intention of (\ref{coi:Proposition:ReduceScalabilityDependency:Equation}) is to reduce the scale dependence of the unbounded variables $y$, $R$ and $S$ of $f$.
If we go through the proof of Proposition \ref{coi:Proposition:ReduceScalabilityDependency} again, we notice that we use the scaling property (\ref{coi:Proposition:Ascaling}) of $A$ for showing (\ref{coi:THomogen:Proof2}). From the proof of (\ref{coi:Proposition:Ascaling}) we notice that this proof only holds, if $y$, $R$ and $S$ are treated as variables and none of them is treated as a constant (since factoring $\kappa^2$ out of a constant, yields an additional factor $\kappa^{-2}$ to the constant). Nevertheless, if one of the variables $y$, $R$ resp.~$S$ is treated as a constant and we set the corresponding value to $y=\zeroVector{m}$, $R=\zeroMatrix{n}$ resp.~$S=\zeroMatrix{n}$, then the proof still holds.
\end{remark}
\begin{example}
\label{AN:Example:CertificateZeroSet}
Consider the variables $R$ and $S$ as constants and set $R=S=\zeroMatrix{n}$. Then
$
T_1(y,z,R,S,\underline{x},\overline{x}):=1
$
does not satisfy (\ref{coi:THomogen}), while
$
T_2(y,z,R,S,\underline{x},\overline{x}):=\lvert y\rvert_{_2}
$
does (cf.~Remark \ref{coi:Remark:ReduceScalabilityDependency}). Note that
$T_2$
violates the requirement of positivity, as demanded in Definition \ref{AN:Definition:NfuncZf}, for $y=\zeroVector{m}$, and hence in this case $f$ is only defined outside the zero set of $T$.
Nevertheless, since the zero set of
$T_2$
has measure zero, it is numerically very unlikely to end up at a point of this zero set and therefore we will also consider this choice of $T$ due to the important fact of the reduction of the scale dependence of $f$ as mentioned in Remark \ref{coi:Remark:ReduceScalabilityDependency}
(also cf.~Subsection \ref{coi:idea:ExclusionBoxesStrategy} (directly after optimization problem (\ref{coi3:Optimierungsproblem})) and Example \ref{example:CSPandCertificateWitT1andTnormy}).
\end{example}
\begin{example}
Choose $n=2$, $m=2$, $c_1=\left(\begin{smallmatrix} 1    \\ -3    \end{smallmatrix}\right)$, $c_2=\left(\begin{smallmatrix} 4    \\  2    \end{smallmatrix}\right)$, $C_1=\left(\begin{smallmatrix} 2 & 0\\  3 & 4\end{smallmatrix}\right)$, $C_2=\left(\begin{smallmatrix}-1 & 0\\ -2 & 7\end{smallmatrix}\right)$, which yields $F_1(x)=2x_1^2+x_1+3x_1x_2-3x_2+4x_2^2$ and $F_2(x)=-x_1^2+4x_1-2x_1x_2+2x_2+7x_2^2$ due to
\cbstartRED(\ref{coi:DefFk})\cbendRED,
$\boldsymbol{x}=[-3,3]\times[-4,4]$ and $\boldsymbol{F}=[-1,7]\times[-2,0]$, then we can illustrate certificate $f$ from (\ref{coi3:Def:f}) with different $T$ from
Example \ref{AN:Example:CertificateZeroSet}
by the following plots\\
~\\
\begin{minipage}{.5\linewidth}
\begin{center}
\captionsetup{type=figure}
\includegraphics[width=5cm]{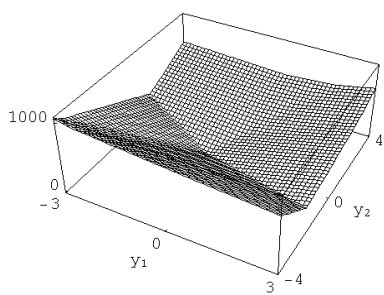}
\captionof{figure}{$f(y_1,y_2)$ for $T_1$}
\end{center}
\end{minipage}
\begin{minipage}{.5\linewidth}
\begin{center}
\captionsetup{type=figure}
\includegraphics[width=5cm]{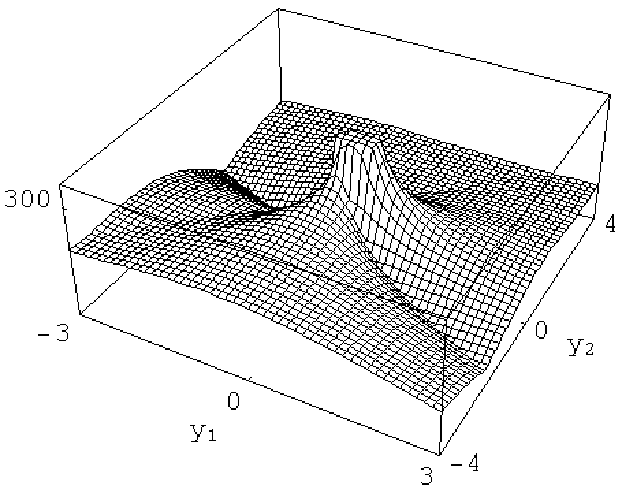}
\captionof{figure}{$f(y_1,y_2)$ for $T_2$}
\end{center}
\end{minipage}
\begin{minipage}{.5\linewidth}
\begin{center}
\captionsetup{type=figure}
\includegraphics[width=5cm]{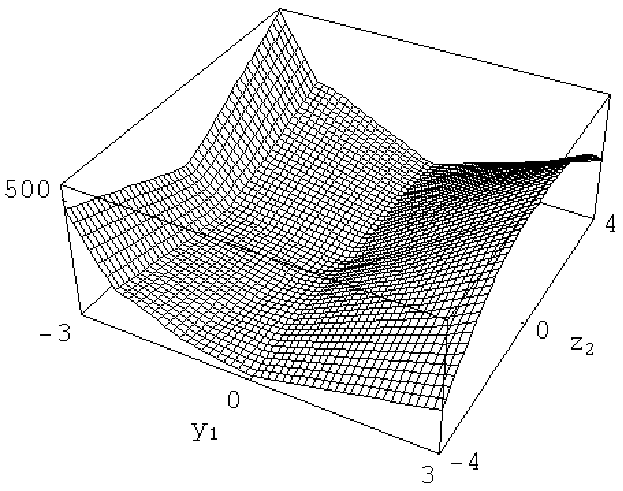}
\captionof{figure}{$f(y_1,z_2)$ for $T_1$}
\end{center}
\end{minipage}
\begin{minipage}{.5\linewidth}
\begin{center}
\captionsetup{type=figure}
\includegraphics[width=5cm]{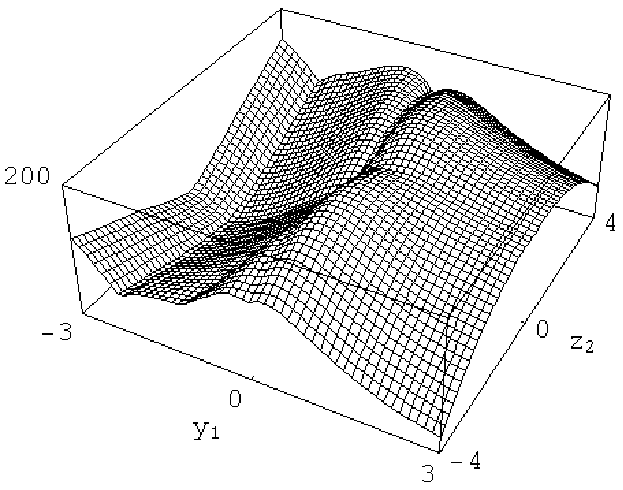}
\captionof{figure}{$f(y_1,z_2)$ for $T_2$}
\end{center}
\end{minipage}
\begin{minipage}{.5\linewidth}
\begin{center}
\captionsetup{type=figure}
\includegraphics[width=5cm]{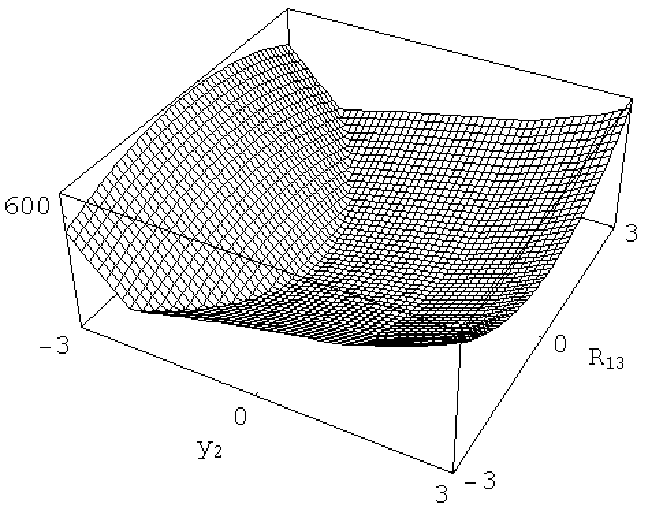}
\captionof{figure}{$f(y_2,R_{13})$ for $T_1$}
\end{center}
\end{minipage}
\begin{minipage}{.5\linewidth}
\begin{center}
\captionsetup{type=figure}
\includegraphics[width=5cm]{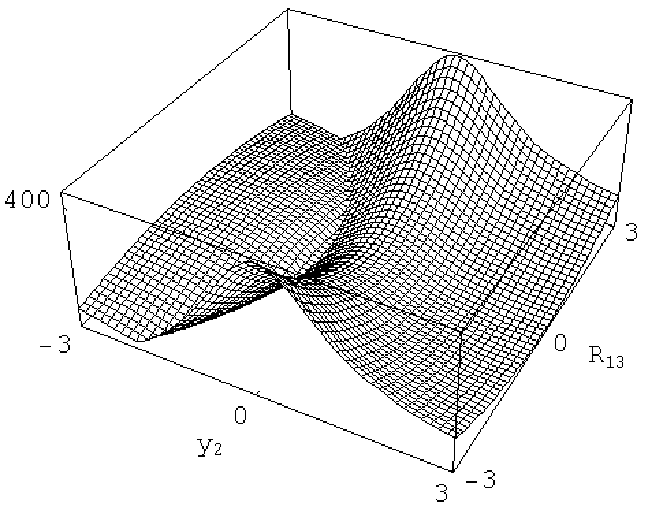}
\captionof{figure}{$f(y_2,R_{13})$ for $T_2$}
\end{center}
\end{minipage}
\begin{minipage}{.5\linewidth}
\begin{center}
\captionsetup{type=figure}
\includegraphics[width=5cm]{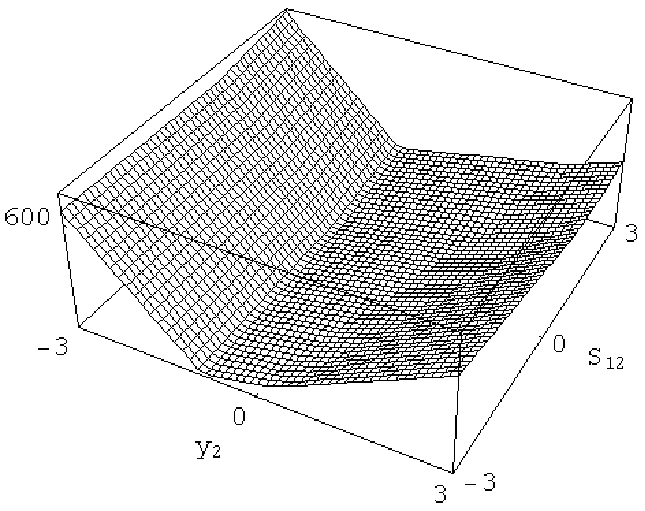}
\captionof{figure}{$f(y_2,S_{12})$ for $T_1$}
\end{center}
\end{minipage}
\begin{minipage}{.5\linewidth}
\begin{center}
\captionsetup{type=figure}
\includegraphics[width=5cm]{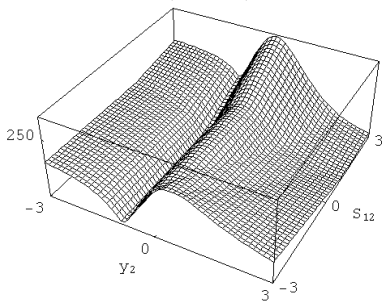}
\captionof{figure}{$f(y_2,S_{12})$ for $T_2$}
\end{center}
\end{minipage}
\end{example}

\subsection{Exclusion boxes for constraint satisfaction problems}
\label{coi:idea:ExclusionBoxesStrategy}
Now we explain in detail how to use Theorem \ref{coi3:Satz:UnzulässigkeitszertifikatVariableIntervallgrenze} for finding exclusion boxes for the CSP (\ref{coi:CSP}).
If we apply a solver for linearly constrained nonsmooth optimization (Note: The certificate $f$ from (\ref{coi3:Def:f}) is not differentiable everywhere due to Remark \ref{COI:Remark:CertificateNumberOfVariablesAndNonsmooth}) to
\begin{equation}
\begin{split}
&\min{f(y,z,\wVariable,u,v)}\cbmathRED\\
&\textnormal{ s.t. }u,v\in\boldsymbol{x}
\komma~
u+r\leq v
\komma~
z\in[u,v]\\
&\hphantom{\textnormal{ s.t. }}y\in\mathbb{R}^m,z,u,v\in\mathbb{R}^n,\wVariable\in\mathbb{R}^{\wDimension}\cbmathRED
\label{coi3:Optimierungsproblem}
\end{split}
\end{equation}
with a fixed
$r\in[\zeroVector{n},\overline{x}-\underline{x}]$ --- although the convergence theory of many solvers (cf., e.g., 
\HauptAlgorithmus by \citet[\GenaueAngabeConvergenceTheory]{HannesPaperB})
requires that all occurring functions are defined on the whole $\mathbb{R}^{\NdimCer}$, which might be violated for certain choices of $T$ (cf.~Example \ref{AN:Example:CertificateZeroSet}) --- and if there occurs a function value smaller than zero (during the optimization process), then there is no feasible point in $[u,v]$ according to Theorem \ref{coi3:Satz:UnzulässigkeitszertifikatVariableIntervallgrenze} and consequently we can reduce the box $\boldsymbol{x}$ to the set $\boldsymbol{x}\setminus[u,v]$ in the CSP (\ref{coi:CSP}).

If $[u,v]=\boldsymbol{x}$ (i.e.~$u$ and $v$ are fixed and therefore no variables), then we can reduce the box $\boldsymbol{x}$ to the empty set, i.e.~the reduction of a box to the empty set is equivalent to removing the box.

The constant $r$ determines the size of the box $[u,v]$, which should be excluded: The closer $r$ is to $\zeroVector{n}$, the smaller the box $[u,v]$ can become (if $r=\zeroVector{n}$, $[u,v]$ can become thin, what we want to prevent, since we want to remove a preferably large box $[u,v]$ out of $\boldsymbol{x}$, as then the remaining set $\boldsymbol{x}\setminus[u,v]$ is preferably small).

If during the optimization a point $z\in\boldsymbol{x}$ is found with $F(z)\in\boldsymbol{F}$, then we have found a feasible point and therefore we can stop the optimization, since then we cannot verify infeasibility for the box $\boldsymbol{x}$.
\begin{remark}
If $\boldsymbol{y}\subseteq\boldsymbol{x}$ and we remove $\boldsymbol{y}$ from $\boldsymbol{x}$, then the remaining set $\boldsymbol{x}\setminus\boldsymbol{y}$ is not closed. Nevertheless, if we just remove $\boldsymbol{y}^{\circ}\subset\boldsymbol{y}$, then the remaining set $\boldsymbol{x}\setminus\boldsymbol{y}^{\circ}\supset\boldsymbol{x}\setminus\boldsymbol{y}$ is closed (i.e.~we remove a smaller box and therefore the remaining set is a bit larger, since it contains the boundary of $\boldsymbol{y}$).
Furthermore, the set $\boldsymbol{x}\setminus\boldsymbol{y}$ can be represented as a union of at most $2n$ $n$-dimensional boxes, i.e.~in particular the number of boxes obtained by this splitting process is linear in $n$.
\end{remark}
We make the assumption that the certificate of infeasibility from (\ref{coi3:Def:f}) of the box $[\hat{u},\hat{v}]$ satisfies
$
f(\hat{y},\hat{z},\hat{\wVariable},\hat{u},\hat{v})=:\hat{\delta}<0
$,
i.e.~$[\hat{u},\hat{v}]$ is an exclusion box according to Theorem \ref{coi3:Satz:UnzulässigkeitszertifikatVariableIntervallgrenze}. For $\delta\in[\hat{\delta},0)$ and a box $\boldsymbol{x}$ with $[\hat{u},\hat{v}]\subseteq\boldsymbol{x}$, we can try to apply a solver for nonlinearly constrained nonsmooth optimization to
\begin{equation}
\begin{split}
&\min{b(y,z,\wVariable,u,v)}\cbmathRED\\
&\textnormal{ s.t. }f(y,z,\wVariable,u,v)\leq\delta\cbmathRED\\
&\hphantom{\textnormal{ s.t. }}u,v\in\boldsymbol{x}
\komma~
u\leq\hat{u}
\komma~
\hat{v}\leq v
\komma~
z\in[u,v]\\
&\hphantom{\textnormal{ s.t. }}y\in\mathbb{R}^m,z,u,v\in\mathbb{R}^n,\wVariable\in\mathbb{R}^{\wDimension}\cbmathRED
\label{coi3:OptimierungsproblemNB}
\end{split}
\end{equation}
to enlarge the exclusion box $[u,v]$ in $\boldsymbol{x}$, where $b:\mathbb{R}^{\NdimCer}\longrightarrow\mathbb{R}$ is a measure for the box $[u,v]$ in the following sense:
If $b:\mathbb{R}^{\NdimCer}\longrightarrow\Rneg$, then the following conditions must hold: If $[u,v]$ is small, then $b(.,u,v)$ is close to $0$ and if $[u,v]$ is large, then $b(.,u,v)$ is negative and large. This means: The larger the box $[u,v]$ is, the more negative $b(.,u,v)$ must be. For examples of this type of box measure cf.~(\ref{Ausschlussboxen:Def:boxMeasureNegative}).
Alternatively,
if
$b:\mathbb{R}^{\NdimCer}\longrightarrow\Rpos$, then the following condition must hold: If $[u,v]\subseteq\boldsymbol{x}$ is close to $\boldsymbol{x}$, then $b(.,u,v)$ is close to $0$. For examples of this type of box measure cf.~(\ref{Ausschlussboxen:Def:boxMeasurePositive}).
\begin{remark}
In opposite to (\ref{coi3:Optimierungsproblem}), where the linear constraint $u+r\leq v$ occurs, we use in (\ref{coi3:OptimierungsproblemNB}) the bound constraints $u\leq\hat{u}$ and $\hat{v}\leq v$.

Furthermore,
we
make the following
two
observations for the global optimization problem
\begin{equation}
\begin{split}
&\min_{x}{F_{\mathrm{obj}}(x)}\\
&          \textnormal{ s.t. } F_i(x)\in\boldsymbol{F}_i~~~\forall i=2,\dots,m\\
&\hphantom{\textnormal{ s.t. }}  x \in\boldsymbol{x}\komma
\label{AN:Remark:CertificateForGlobalOptimizationProblem}
\end{split}
\end{equation}
where $F_{\mathrm{obj}}:\mathbb{R}^n\longrightarrow\mathbb{R}$:
First of all,
the
certificate $f$ from (\ref{coi3:Def:f}) can be used for finding exclusion boxes in the global optimization problem (\ref{AN:Remark:CertificateForGlobalOptimizationProblem}) with an arbitrary objective function $F_{\mathrm{obj}}$, since the certificate $f$ only depends on the constraint data $F$, $\boldsymbol{F}$ and $\boldsymbol{x}$ (cf.~the CSP (\ref{coi:CSP})) and since a solution of an optimization problem is necessarily feasible.
Secondly,
we
denote the current lowest known function value of
the
optimization problem
(\ref{AN:Remark:CertificateForGlobalOptimizationProblem})
by $F_{\mathrm{obj}}^{\mathrm{cur}}$. Now, if we can find a box $[u,v]\subseteq{x}$ for which the certificate $f$ from (\ref{coi3:Def:f}) with
$\boldsymbol{F}_1:=[-\infty,F_{\mathrm{obj}}^{\mathrm{cur}}]$
has a negative value, then Theorem \ref{coi3:Satz:UnzulässigkeitszertifikatVariableIntervallgrenze} implies that for all $x\in[u,v]$ there exists $k\in\lbrace1,\dots,m\rbrace$ with $F_k(x)\not\in\boldsymbol{F}_k$, which is equivalent that for all $x\in[u,v]$
we have
$F_{\mathrm{obj}}^{\mathrm{cur}}<F_1(x)$ or there exists $i\in\lbrace2,\dots,m\rbrace$ with $F_i(x)\not\in\boldsymbol{F}_i$, i.e.~any point in the box $[u,v]$ has an objective function value which is higher than the current lowest known function value $F_{\mathrm{obj}}^{\mathrm{cur}}$ or is infeasible. Consequently, the box $[u,v]$ cannot contain a feasible point with function value lower equal $F_{\mathrm{obj}}^{\mathrm{cur}}$, and hence the box $[u,v]$ cannot contain a global minimizer of the global optimization problem (\ref{AN:Remark:CertificateForGlobalOptimizationProblem}). Therefore we can exclude the box $[u,v]$ from further consideration.
\end{remark}
\begin{example}
For measuring the box $[u,v]$ with $b:\mathbb{R}^{\NdimCer}\longrightarrow\Rneg$, we can use any negative $p$-norm with $p\in[1,\infty]$ as well as variants of them
\begin{equation}
\begin{aligned}
b_-^1(y,z,\wVariable,u,v)
&:=
-\lvert v-u\rvert_{_1}
\cbmathRED
\komma\qquad
b_-^2(y,z,\wVariable,u,v)
:=
-\tfrac{1}{2}\lvert v-u\rvert_{_2}^2
\cbmathRED
\komma~\\
b_-^{\infty}(y,z,\wVariable,u,v)
&:=
-\lvert v-u\rvert_{_{\infty}}
\punkt
\cbmathRED
\label{Ausschlussboxen:Def:boxMeasureNegative}
\end{aligned}
\end{equation}
For measuring the box $[u,v]$ with $b:\mathbb{R}^{\NdimCer}\longrightarrow\Rpos$, we can use any $p$-norm with $p\in[1,\infty]$ as well as variants of them
\begin{equation}
\begin{aligned}
b_+^1(y,z,\wVariable,u,v)
&:=
\lvert
\left(
\begin{smallmatrix}
u-\underline{x}\\
v-\overline{x}
\end{smallmatrix}
\right)
\rvert_{_1}
\cbmathRED
\komma\qquad
b_+^2(y,z,\wVariable,u,v)
:=
\tfrac{1}{2}
\lvert
\left(
\begin{smallmatrix}
u-\underline{x}\\
v-\overline{x}
\end{smallmatrix}
\right)
\rvert_{_2}^2
\cbmathRED
\komma~\\
b_+^{\infty}(y,z,\wVariable,u,v)
&:=
\lvert
\left(
\begin{smallmatrix}
u-\underline{x}\\
v-\overline{x}
\end{smallmatrix}
\right)
\rvert_{_{\infty}}
\punkt
\cbmathRED
\label{Ausschlussboxen:Def:boxMeasurePositive}
\end{aligned}
\end{equation}
$b_-$ is concave, while $b_+$ is convex. $b_-$ can be used for unbounded $\boldsymbol{x}$, while this is not possible for $b_+$. $b^2$ is smooth, while $b^1$ and $b^{\infty}$ are not differentiable.
$b^1$ has an equal growing rate for all components. The growing rate of $b^2$ depends on $\sgn\big(\tfrac{1}{2}\lvert.\rvert_{_2}^2-1\big)$. $b^{\infty}$ already grows, if the absolute value of the largest components grows.
\end{example}
\begin{example}
\label{example:CSPandCertificateWitT1andTnormy}
Choose $n=1$, $m=1$,
$\wVariable=0$,
$c_1=1$, $C_1=\tfrac{1}{2}$, which yields $F(x)=\tfrac{1}{2}x^2+x$ due to (\ref{coi:DefFk}), as well as $\boldsymbol{x}=[-1,2]$ and consider two CSPs (\ref{coi:CSP}) with $\boldsymbol{F}^1=[-2, 1]$ resp.~$\boldsymbol{F}^2=[-2,-1]$ which yield the following two graphics\\
\begin{minipage}{.5\linewidth}
\begin{center}
\captionsetup{type=figure}
\includegraphics[width=3cm]{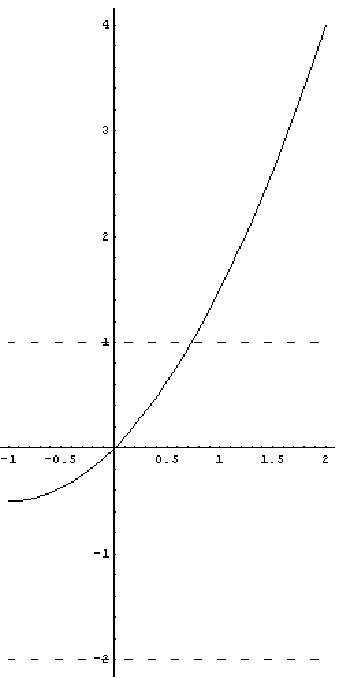}
\captionof{figure}{$F$ for $\boldsymbol{F}^1$}
\end{center}
\end{minipage}
\begin{minipage}{.5\linewidth}
\begin{center}
\captionsetup{type=figure}
\includegraphics[width=3cm]{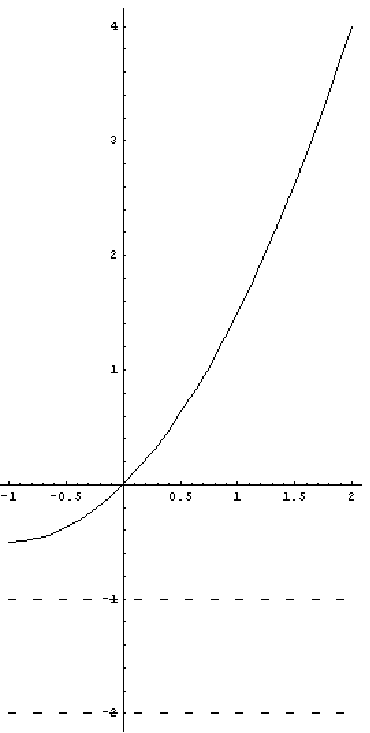}
\captionof{figure}{$F$ for $\boldsymbol{F}^2$}
\end{center}
\end{minipage}
from which we can see that the CSP has feasible points for $\boldsymbol{F}^1$, while it is infeasible for $\boldsymbol{F}^2$. The corresponding certificates $f$ from (\ref{coi3:Def:f}), where we only consider the variables $y\in\mathbb{R}$ and $z\in\boldsymbol{x}$ as well as different $T$ from
Example \ref{AN:Example:CertificateZeroSet},
and where we denote the function value of a local minimizer of the optimization problem (\ref{coi3:Optimierungsproblem}) by $\hat{f}$, can be illustrated by the following plots\\
\begin{minipage}{.5\linewidth}
\begin{center}
\captionsetup{type=figure}
\includegraphics[width=5cm]{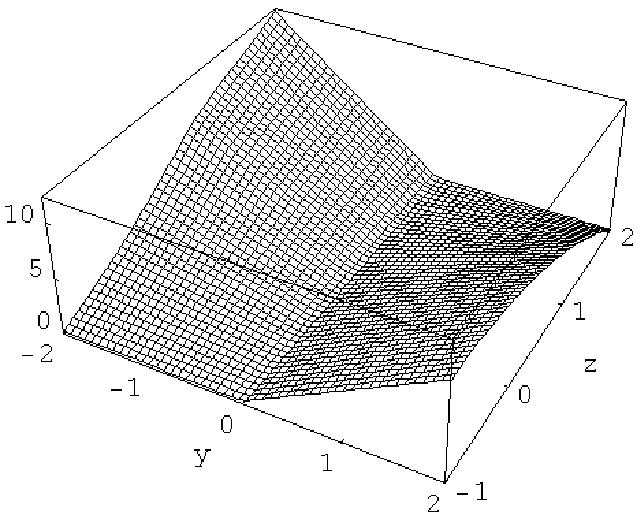}
\captionof{figure}{$f$ for $(\boldsymbol{F}^1$,$T_1)$ $\Longrightarrow$ $\hat{f}=0$}
\end{center}
\end{minipage}
\begin{minipage}{.5\linewidth}
\begin{center}
\captionsetup{type=figure}
\includegraphics[width=5cm]{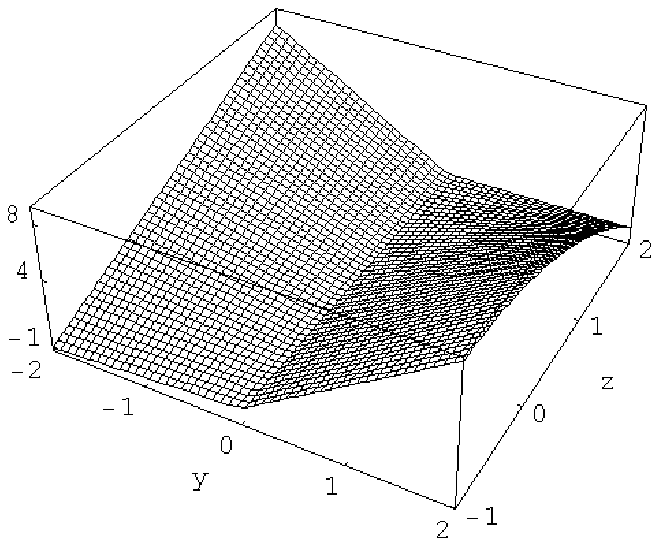}
\captionof{figure}{$f$ for $(\boldsymbol{F}^2,T_1)$ $\Longrightarrow$ $\hat{f}=-1$}
\end{center}
\end{minipage}
\begin{minipage}{.5\linewidth}
\begin{center}
\captionsetup{type=figure}
\includegraphics[width=5cm]{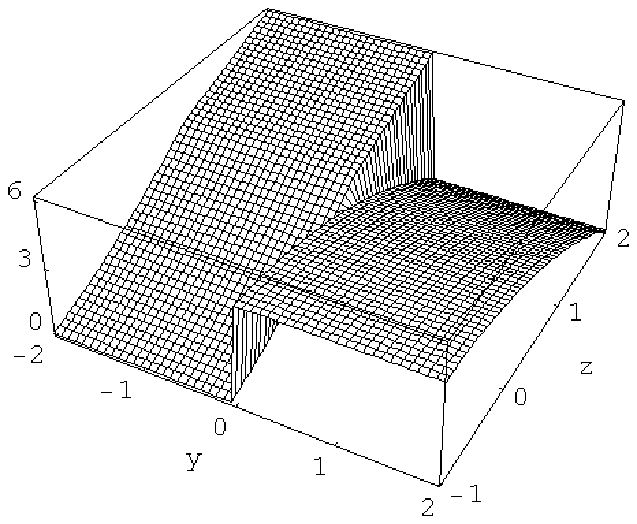}
\captionof{figure}{$f$ for $(\boldsymbol{F}^1,T_2)$ $\Longrightarrow$ $\hat{f}=0$}
\label{Summary02Bsp01FeasibleCertificateNorm}
\end{center}
\end{minipage}
\begin{minipage}{.5\linewidth}
\begin{center}
\captionsetup{type=figure}
\includegraphics[width=5cm]{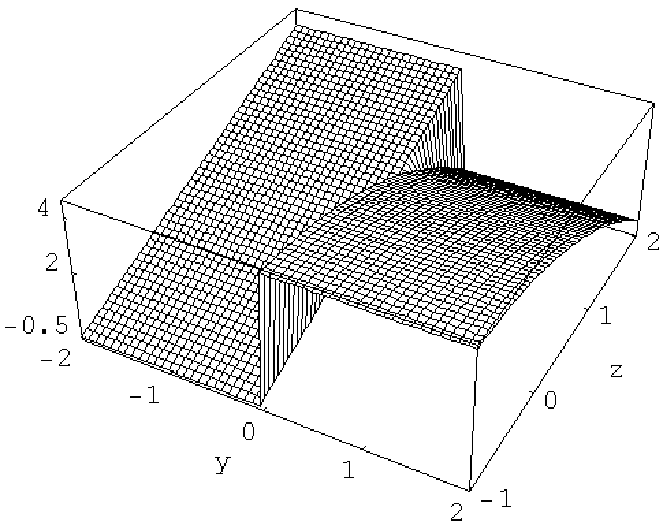}
\captionof{figure}{$f$ for $(\boldsymbol{F}^2,T_2)$ $\Longrightarrow$ $\hat{f}=-\tfrac{1}{2}\punkt$}
\label{Figure:Summary02Bsp01InfeasibleCertificateNorm}
\end{center}
\end{minipage}
We see from Figure \ref{Summary02Bsp01FeasibleCertificateNorm} and Figure \ref{Figure:Summary02Bsp01InfeasibleCertificateNorm} that the certificate $f$ is not defined for $y=0$ due to the definition of $T_2$ in
Example \ref{AN:Example:CertificateZeroSet}.
\end{example}
\section{Starting point}
\label{Paper:ImplementationIssuesForFindingExclusionBoxes}
We implemented the suggestions from 
Subsection
\ref{coi:idea:ExclusionBoxesStrategy} in GloptLab by \citet{Domes} which is a configurable MATLAB framework for computing the 
global solution of a 
quadratic
CSP (\ref{coi:CSP}), i.e.~with $F_k$ from (\ref{coi:DefFk}). The matrices $C_k\in\mathbb{R}^{n\times n}$ are lower triangular in GloptLab
\begin{equation}
C_k\in\Tril{n}
:=
\lbrace
A\in\mathbb{R}^{n\times n}:A_{ij}=0\textnormal{ for }i<j
\rbrace
\punkt
\label{coi:VorUntereDreiecksmatrix}
\end{equation}
For running GloptLab the MATLAB toolbox INTLAB by \citet{Rump}, \texttt{lp\_solve} by \citet{Berkelaar}, SEDUMI by \citet{Sturm,Polik} as well as SDPT3 by \citet{SDPT3} were installed for using all features of GloptLab.
\cbstartRED

So the last issue that remains to be discussed is, how to
find a point $(y,z,\wVariable,u,v)$ being feasible for the linearly constrained optimization problem (\ref{coi3:Optimierungsproblem}) with
$f(y,z,\wVariable,u,v)<0$
quickly. For this we need a good starting point $(y^0,z^0,\wVariable^0,u^0,v^0)$ and therefore we must take the following observations into account:
$y^0$ and $z^0$ should be chosen so, that $\Nfunc(y^0,z^0)$ is positive,
and
$(y^0,z^0,\wVariable^0,u^0,v^0)$ should be chosen so, that the term $Z(y^0,z^0,\wVariable^0,u^0,v^0)$, which is non-negative due to (\ref{coi3:SatzZnichtnegativ}), is near zero.
These facts lead to the following suggestions for choosing a starting point $(y^0,z^0,\wVariable^0,u^0,v^0)$\cbendRED:
First of all,
if
the solver in use can only handle strictly feasible bound/linear constraints (e.g., \HauptAlgorithmus by \citet[\GenaueAngabeFive]{HannesPaperA} with using \texttt{socp} by \citet{BoydSOCP1995} for computing the search direction),
then the initial choices of $u^0$ and $v^0$ must satisfy $u^0,v^0\in(\underline{x},\overline{x})$ and $u^0+r<v^0$, e.g.,
$u^0:=(1-t_0)\underline{x}+t_0(\overline{x}-r)$
and
$v^0:=(1-t_1)\underline{x}+t_1(\overline{x}-r)$
for some fixed $t_0,t_1\in(0,1)$ with $t_0<t_1$.
Otherwise (e.g., SolvOpt by \citet{SolvOptPublication} or \HauptAlgorithmus with using MOSEK for computing the search direction) we take the endpoints of $\boldsymbol{x}$ for $u^0$ and $v^0$.
Secondly,
the
natural choice for the starting value of $z\in[u,v]\subseteq\boldsymbol{x}$ is the midpoint
$
z^0:=\tfrac{1}{2}(u^0+v^0)
$
of the box $[u^0,v^0]$.
Thirdly,
to
get the term $\max{\big(0,\Nfunc(y,z)\big)}$ in the certificate $f$ from (\ref{coi3:Def:f}) as large as possible, we make the following choices:
For the case $\underline{F}_k=-\infty$ resp.~the case $\overline{F}_k=\infty$ resp.~the case that both $\underline{F}_k$ and $\overline{F}_k$ are finite, we choose
\begin{equation}
\begin{split}
y_k^0:=
\left\lbrace
\begin{array}{ll}
          - 1 & \textnormal{if }\overline{F}_k<F_k(z^0)\\
\hphantom{-}0 & \textnormal{else}
\end{array}
\right.
\komma~
y_k^0:=
\left\lbrace
\begin{array}{ll}
1 & \textnormal{if }F_k(z^0)<\underline{F}_k\\
0 & \textnormal{else}
\end{array}
\right.
\komma~
y_k^0:=
\left\lbrace
\begin{array}{ll}
\hphantom{-}1 & \textnormal{if }F_k(z^0)<\underline{F}_k\\
          - 1 & \textnormal{if }\overline{F}_k<F_k(z^0)\\
\hphantom{-}0 & \textnormal{else}
\end{array}
\right.
\label{coi:yStartpunktFLunboundedCOMPACT}
\end{split}
\end{equation}
respectively due to (\ref{coi:SatzNexplizit}) and (\ref{coi:Proposition:NyzEqualsMinusInfty}).
Finally,
for
the choices of $R$ and $S$ we refer to Proposition \ref{coi:StartinPoint:Proposition:Matrices}.
\begin{remark}
If we choose $T=T_2$ (cf.~Example \ref{AN:Example:CertificateZeroSet}),
then
\cbstartRED
$T(y,z,\wVariable,u,v)=0~\Longleftrightarrow~y=\zeroVector{m}$.
\cbendRED
Therefore, if $y^0=\zeroVector{m}$ occurs as starting point, then we have a feasible point $F(z)\in\boldsymbol{F}$ due to
(\ref{coi:yStartpunktFLunboundedCOMPACT}).
Furthermore, we can expect that no solver should have difficulties with this choice of $T$ because of the small size of the zero set of $T$ due to Example \ref{AN:Example:CertificateZeroSet}.
\end{remark}
In the following we will make use of the MATLAB operators $\textnormal{diag}$, $\textnormal{tril}$ and $\textnormal{triu}$.
\begin{proposition}
\label{coi:StartinPoint:Proposition:Matrices}
Let $F_k$ be quadratic and let \refh{coi:VorUntereDreiecksmatrix} be satisfied. Choose any $y\in\mathbb{R}^m$ and consider the modified Cholesky factorization
\begin{equation}
\hat{A}=\hat{R}^T\hat{R}-D
\label{coi:StartinPoint:Proposition:Matrices:AhatChol}
\end{equation}
of $\hat{A}$ (with $\hat{R}\in\Triu{n}$ and the non-negative diagonal matrix $D\in\mathbb{R}^{n\times n}$), where
\begin{equation}
\hat{A}
:=
C(y)+S^T-S
\in\Sym{n}
\komma\quad
S
:=
-\tfrac{1}{2}\textnormal{triu}\big(C(y)^T,1\big)\in\sTriu{n}
\label{coi:StartinPoint:Proposition:Matrices:AhatCOMPACT}
\end{equation}
and $\Sym{n}$ denotes the space of all symmetric $n\times n$-matrices.
Then
\begin{equation}
\hat{A}
=
C(y)-\tfrac{1}{2}\textnormal{tril}\big(C(y),-1\big)+\tfrac{1}{2}\textnormal{triu}\big(C(y)^T,1\big)
\komma\quad
\textnormal{diag}(\hat{A})
=
\textnormal{diag}\big(C(y)\big)
\punkt
\label{coi:StartinPoint:Proposition:Matrices:AhatAlternativeRepresentationCOMPACT}
\end{equation}
Furthermore,
if
we set
\begin{equation}
R:=D^{\frac{1}{2}}
\komma
\label{coi:StartinPoint:Proposition:Matrices:R}
\end{equation}
then
$
A(y,R,S)=\hat{R}^T\hat{R}\in\Sym{n}
$.
\end{proposition}
\begin{proof}
Since $F_k$ is quadratic by assumption, the statements of Proposition \ref{coi:corollary:OriginalCertificate} hold.
Since $C(y)$ is lower triangular due to (\ref{coi:VorUntereDreiecksmatrix}) and (\ref{coi:DefCCOMPACT}), we obtain $S\in\sTriu{n}$ and
\begin{align}
\hat{A}
=
C(y)-\tfrac{1}{2}\Big(\textnormal{triu}\big(C(y)^T,1\big)\Big)^T+\tfrac{1}{2}\textnormal{triu}\big(C(y)^T,1\big)
\label{coi:StartinPoint:Proposition:Matrices:Proof1}
\end{align}
due to (\ref{coi:StartinPoint:Proposition:Matrices:AhatCOMPACT}).
Now, (\ref{coi:StartinPoint:Proposition:Matrices:Proof1}) implies (\ref{coi:StartinPoint:Proposition:Matrices:AhatAlternativeRepresentationCOMPACT}).
We calculate
\begin{equation}
C(y)^T-\tfrac{1}{2}\textnormal{triu}\big(C(y)^T,1\big)
=
\textnormal{diag}\big(C(y)\big)
+\tfrac{1}{2}\textnormal{triu}\big(C(y)^T,1\big)
\punkt
\label{coi:StartinPoint:Proposition:Matrices:Proof2}
\end{equation}
Because of
$
\tfrac{1}{2}
\big(
\textnormal{triu}\big(C(y)^T,1\big)
\big)^T
=
\textnormal{tril}\big(C(y),-1\big)
-\tfrac{1}{2}
\big(
\textnormal{triu}\big(C(y)^T,1\big)
\big)^T
$
we have
$
\textnormal{diag}\big(C(y)\big)
+\tfrac{1}{2}
\big(
\textnormal{triu}\big(C(y)^T,1\big)
\big)^T
=
C(y)
-\tfrac{1}{2}
\big(
\textnormal{triu}\big(C(y)^T,1\big)
\big)^T
$.
Therefore,
combining 
(\ref{coi:StartinPoint:Proposition:Matrices:Proof1}) and
(\ref{coi:StartinPoint:Proposition:Matrices:Proof2})
yields
$
\hat{A}^T
=
\hat{A}
$,
i.e.~$\hat{A}\in\Sym{n}$.
Consequently
there exists a modified Cholesky factorization of $\hat{A}$ of the form (\ref{coi:StartinPoint:Proposition:Matrices:AhatChol}). Hence, we can choose $R$ according to (\ref{coi:StartinPoint:Proposition:Matrices:R}) and evaluating $A$ at $(y,R,S)$ with $R$ from (\ref{coi:StartinPoint:Proposition:Matrices:R}) and $S$ from (\ref{coi:StartinPoint:Proposition:Matrices:AhatCOMPACT}) yields $A(y,R,S)=\hat{R}^T\hat{R}$ due to
(\ref{coi:DefCCOMPACT}),
(\ref{coi:StartinPoint:Proposition:Matrices:R}) and
(\ref{coi:StartinPoint:Proposition:Matrices:AhatChol}).
\qedhere
\end{proof}
\begin{remark}
If $\hat{A}$ is positive semidefinite, then $D=\zeroMatrix{n}$ due to (\ref{coi:StartinPoint:Proposition:Matrices:AhatChol}).
If $C(y)$ is a diagonal matrix, then $S=\zeroMatrix{n}$ due to (\ref{coi:StartinPoint:Proposition:Matrices:AhatCOMPACT}).
Due to (\ref{coi:StartinPoint:Proposition:Matrices:AhatAlternativeRepresentationCOMPACT}), we can construct $\hat{A}$ by setting $\hat{A}$ equal to $C(y)$, then multiplying 
the lower triangular part of $\hat{A}$
by $\tfrac{1}{2}$ and finally copying the resulting lower triangular part of $\hat{A}$ to the upper triangular part of $\hat{A}$.
\end{remark}
Now we combine the facts that we presented in this subsection
to the Algorithm \ref{coi:Alg:Startpunkt},
which we will use for creating a starting point for the optimization problem (\ref{coi3:Optimierungsproblem})
\cbstartRED
with quadratic $F$\cbendRED.
\begin{algorithm}\label{coi:Alg:Startpunkt}~\\
\texttt{if the solver can only handle strictly feasible bound/linear constraints}\\
\hphantom{aaa}\texttt{Choose} $0<t_0<t_1<1$ (e.g., $t_0=0.1$ and $t_1=0.9$)\\
\hphantom{aaa}$u=(1-t_0)\underline{x}+t_0(\overline{x}-r)$\\
\hphantom{aaa}$v=(1-t_1)(\underline{x}+r)+t_1\overline{x}$\\
\texttt{else}\\
\hphantom{aaa}$u=\underline{x}$\\
\hphantom{aaa}$v=\overline{x}$\\
\texttt{end}\\
$z=\tfrac{1}{2}(u+v)$\\
$F=F(z)$\\
\texttt{if }$F\in\boldsymbol{F}$\\
\hphantom{aaa}\texttt{stop (found feasible point => cannot verify infeasibility)}\\
\texttt{end if}\\
\texttt{for }$k=1:m$\\
\hphantom{aaa}\texttt{if }$\underline{F}_k=-\infty$\\
\hphantom{aaaaaa}\texttt{if }$\overline{F}_k<F_k$\\
\hphantom{aaaaaaaaa}$y_k=-1$\\
\hphantom{aaaaaa}\texttt{else}\\
\hphantom{aaaaaaaaa}$y_k=0$\\
\hphantom{aaaaaa}\texttt{end if}\\
\hphantom{aaa}\texttt{else if }$\overline{F}_k=\infty$\\
\hphantom{aaaaaa}\texttt{if }$F_k<\underline{F}(k)$\\
\hphantom{aaaaaaaaa}$y_k=1$\\
\hphantom{aaaaaa}\texttt{else}\\
\hphantom{aaaaaaaaa}$y_k=0$\\
\hphantom{aaaaaa}\texttt{end if}\\
\hphantom{aaa}\texttt{else}\\
\hphantom{aaaaaa}\texttt{if }$F_k<\underline{F}_k$\\
\hphantom{aaaaaaaaa}$y_k=1$\\
\hphantom{aaaaaa}\texttt{else if }$\overline{F}_k<F_k$\\
\hphantom{aaaaaaaaa}$y_k=-1$\\
\hphantom{aaaaaa}\texttt{else}\\
\hphantom{aaaaaaaaa}$y_k=0$\\
\hphantom{aaaaaa}\texttt{end if}\\
\hphantom{aaa}\texttt{end if}\\
\texttt{end for}\\
\texttt{Compute }C(y)\\
$S=-\tfrac{1}{2}$\texttt{triu}$\big(C(y)^T,1\big)$\\  
$[\hat{R},D]=$\texttt{modified\_cholesky\_factorization}$\big(C(y)+S^T-S\big)$\\  
$R=$\texttt{sqrt}$(D)$
\end{algorithm}
\begin{remark}
Infeasible constrained solvers (e.g., SolvOpt by \citet{SolvOptPublication}) can be applied directly to the nonlinearly constrained optimization problems (\ref{coi3:OptimierungsproblemNB}). In this case the starting point created by Algorithm \ref{coi:Alg:Startpunkt} can be used at once without solving optimization problem (\ref{coi3:Optimierungsproblem}) first as it is necessary for \HauptAlgorithmus by \citet[\GenaueAngabeSix]{HannesPaperA}. Therefore, the bound constraints $u\leq\hat{u}$, $\hat{v}\leq v$ of optimization problem (\ref{coi3:OptimierungsproblemNB}) do not occur in this situation.
Nevertheless, it is useful in this case to add the linear constraint $u+r\leq v$ (with a fixed $r>\zeroVector{\NdimCer}$) from optimization problem (\ref{coi3:Optimierungsproblem}) to the constrained problem for preventing the box $[u,v]$ from becoming too small.
\end{remark}

\section{Numerical results}
In the following section we compare the numerical results of \HauptAlgorithmus by \citet[\GenaueAngabeANone]{HannesPaperA}, MPBNGC by \citet{MPBNGC} and SolvOpt by \citet{SolvOptPublication}
for some
examples that arise in the context of finding exclusion boxes for a quadratic CSP in GloptLab by \citet{Domes}.
\subsection{Introduction}
We will make tests for
\begin{itemize}
\item (the reduced version of) \HauptAlgorithmus for nonsmooth, nonconvex optimization problems with inequality constraints by \citet[\GenaueAngabeANone]{HannesPaperA} (with optimality tolerance $\varepsilon:=10^{-5}$
and with MOSEK by \citet{AndersenPaper} as QCQP-solver for determining the search direction), where we refer to the linearly constrained version as ``BNLC''
and to the nonlinearly constrained version as
``Red(uced) Alg(orithm)''.
It is an extension of the bundle-Newton method for nonsmooth, nonconvex unconstrained minimization by \citet{Luksan,LuksanPBUNandBNEW} to the nonlinearly constrained problems.
\item MPBNGC by \citet{MPBNGC} (with standard termination criterions; 
since MPBNGC turned out to be very fast with respect to pure solving time for the low dimensional examples in the case of successful termination with a stationary point, the number of iterations and function evaluations was chosen in a way that in the other case the solving times of the different algorithms have approximately at least the same magnitude)
\item SolvOpt by \citet{SolvOptPublication} (with the standard termination criterions, which are described in \citet{SolvOpt})
\end{itemize}
(we choose MPBNGC and SolvOpt for our comparisons, since both are written in a compiled programming language, both are publicly available, and both support nonconvex constraints) on the following examples:
\begin{itemize}
\item We give results for the linearly constrained optimization problem (\ref{coi3:Optimierungsproblem}) with a fixed box (i.e.~without optimizing $u$ and $v$) for dimensions between $4$ and $11$ in
Subsection
\ref{section:ExclusionBoxTests:LinearlyConstrainedCaseFixedInterval}.
\item We give results for the linearly constrained optimization problem (\ref{coi3:Optimierungsproblem}) with a variable box (i.e.~with optimizing $u$ and $v$) for dimensions between $8$ and $21$ in
Subsection
\ref{section:ExclusionBoxTests:LinearlyConstrainedCaseVariableInterval}.
\item We give results for the nonlinearly constrained optimization problem (\ref{coi3:OptimierungsproblemNB}) for dimension $8$ in
Subsection
\ref{section:ExclusionBoxTests:NonLinearlyConstrainedCase}, where we use
$
b_+^1(y,z,R,S,u,v)
:=
\lvert
\left(
\begin{smallmatrix}
u-\underline{x}\\
v-\overline{x}
\end{smallmatrix}
\right)
\rvert_{_1}
$
as the objective function.
\end{itemize}
The underlying data for these nonsmooth optimization problems was extracted from real CSPs that occur in GloptLab by \citet{Domes}. Apart from $u$ and $v$, we will concentrate on the optimization of the variables $y$ and $z$ due to the large number of tested examples (cf.~Subsection \ref{subsection:PureSolvingTime}), and since the additional optimization of $R$ and $S$ did not have much impact on the quality of the results which was discovered in additional empirical observations, where a detailed analysis of these observations goes beyond the scope of this
paper.
Furthermore, we will make our tests for the two different choices of the function $T$ from
Example \ref{AN:Example:CertificateZeroSet},
which occurs in the denominator of the certificate $f$ from (\ref{coi3:Def:f}), where for the latter one $f$ is only defined outside of the zero set of $T$ which has measure zero.

The (extensive) tables corresponding to the results, which we will discuss in this section, can be found in \citet[\GenaueAngabeANtwo]{HannesPaperA}.

All test examples will be sorted with respect to the problem dimension (beginning with the smallest). Furthermore, we use analytic derivative information for all occurring functions (Note: Implementing analytic derivative information for the certificate from (\ref{coi3:Def:f}) effectively, is a nontrivial task) and we perform all tests on 
an Intel Pentium IV with 3 GHz and 1 GB RAM running Microsoft Windows XP.

We introduce the following notation for the record of the solution process of an algorithm.
\begin{notation}
We denote the number of performed iterations by Nit,
we
denote the final number of evaluations of function dependent data by
\begin{align*}
\textnormal{Na}
&:=
\textnormal{``Number of calls to }(f,g,G,F,\hat{g},\hat{G})\textnormal{'' (\textnormal{Red Alg})}
\\
\NforMPBNGC
&:=
\textnormal{``Number of calls to }(f,g,F,\hat{g})\textnormal{'' (MPBNGC)}
\\
\NforSolvOpt
&:=
\textnormal{``Number of calls to }(f,F)\textnormal{'' (SolvOpt)}
\\
\textnormal{Ng}
&:=
\textnormal{``Number of calls to }g\textnormal{'' (SolvOpt)}
\\
\textnormal{N}\hat{\textnormal{g}}
&:=
\textnormal{``Number of calls to }\hat{g}\textnormal{'' (SolvOpt)}
\komma
\end{align*}
and we denote the duration of the solution process by
\begin{align*}
\TotalTime
&:=
\textnormal{``Time in milliseconds''}
\\
\PartTime
&:=
\textnormal{``Time in milliseconds (without (QC)QP)''}\komma
\end{align*}
where $\PartTime$ is only relevant for \HauptAlgorithmus.
\end{notation}
\begin{remark}
In particular the percentage of the time spent in the (QC)QP in \HauptAlgorithmus is given by
\begin{equation}
p_1
:=
\tfrac
{
\TotalTime(\textnormal{Red Alg})
-
\PartTime(\textnormal{Red Alg})
}{
\TotalTime(\textnormal{Red Alg})
}
\punkt
\label{NumericalResults:Remark:p1}
\end{equation}
\end{remark}
For comparing the cost of evaluating function dependent data (like e.g., function values, subgradients,\dots) in a
preferably
fair way (especially for solvers that use different function dependent data), we
will
make use of the following realistic ``credit point system'' that an optimal implementation of algorithmic differentiation in backward mode suggests (cf.~\citet{Griewank} and
\citet{schichl2004global,schichlcoconut,SchichlHabilitation}).
\begin{definition}
Let $f_A$, $g_A$ and $G_A$ resp.~$F_A$, $\hat{g}_A$ and $\hat{G}_A$ be the number of function values, subgradients and (substitutes of) Hessians of the objective function resp.~the constraint that an algorithm $A$ used for solving a nonsmooth optimization problem which may have linear constraints and at most one single nonsmooth nonlinear constraint. Then we define the cost of these evaluations by
\begin{equation}
c(A):=f_A+3g_A+3N\cdot G_A+\textnormal{nlc}\cdot(F_A+3\hat{g}_A+3N\cdot\hat{G}_A)
\label{NumericalResults:Definition:cost}
\komma
\end{equation}
where
$\textnormal{nlc}=1$ if the optimization problem has a nonsmooth nonlinear constraint, and $\textnormal{nlc}=0$ otherwise.
\end{definition}
Since the \HauptAlgorithmus evaluates $f$, $g$, $G$ and $F$, $\hat{g}$, $\hat{G}$ at every call that computes function dependent data (cf.~\citet[\GenaueAngabeANthree]{HannesPaperA}), we obtain
\begin{equation*}
c(\textnormal{Red Alg})
=
(1+\textnormal{nlc})\cdot\textnormal{Na}\cdot(1+3+3N)
\punkt
\end{equation*}
Since MPBNGC evaluates $f$, $g$ and $F$, $\hat{g}$ at every call that computes function dependent data (cf.~\citet{MPBNGC}), the only difference to \HauptAlgorithmus with respect to $c$ from (\ref{NumericalResults:Definition:cost}) is that MPBNGC uses no information of Hessians and hence we obtain
\begin{equation*}
c(\textnormal{MPBNGC})
=
(1+\textnormal{nlc})\cdot\NforMPBNGC\cdot(1+3)
\punkt
\end{equation*}
Since SolvOpt evaluates $f$ and $F$ at every call that computes function dependent data and only sometimes $g$ or $\hat{g}$ (cf.~\citet{SolvOpt}), we obtain
\begin{equation*}
c(\textnormal{SolvOpt})
=
(1+\textnormal{nlc})\cdot\NforSolvOpt+3(\textnormal{Ng}+\textnormal{nlc}\cdot\textnormal{N}\hat{g})
\punkt
\end{equation*}

We will visualize the performance of two algorithms $A$ and $B$
in this section
by the following record-plot: In this plot the abscissa is labeled by the name of the test example and the value of the ordinate is given by
$
\rp(c):=c(B)-c(A)
$
(i.e.~if $\rp(c)>0$, then $\rp(c)$ tells us how much better algorithm $A$ is than algorithm $B$ with respect to $c$ for the considered example in absolute numbers; if $\rp(c)<0$, then $\rp(c)$ quantifies the advantage of algorithm $B$ in comparison to algorithm $A$; if $\rp(c)=0$, then both algorithms are equally good with respect to $c$). The scaling of the plots is chosen in a way that plots that contain the same test examples are comparable (although the plots may have been generated by results from different algorithms).
\subsection{Overview of the results}
\label{subsection:PureSolvingTime}
We compare the total time $t_1$ of the solution process
\begin{center}
\begin{tiny}
\begin{tabular}{l|rrrrr|}
                                  & \multicolumn{1}{c}{\TotalTime(Red Alg)}
                                  & \multicolumn{1}{c}{\PartTime(Red Alg)}
                                  & \multicolumn{1}{c}{$p_1$}
                                  & \multicolumn{1}{c}{\TotalTime(MPBNGC)}
                                  & \multicolumn{1}{c|}{\TotalTime(SolvOpt)}
\\
\hline
                                     & \multicolumn{5}{c|}{$T=1$}\\
Linearly constrained (fixed box)     &  1477 &  215 & 0.85 &   231 &  2754 \\
Linearly constrained (variable box)  &   782 &   60 & 0.92 &    30 &  1546 \\
Nonlinearly constrained              & 25420 & 4885 & 0.81 & 21860 &  38761\\
Nonlinearly constrained (*)          & 19053 & 3723 & 0.80 &  2067 &  30312\\
\hline
                                     & \multicolumn{5}{c|}{$T=\lvert y\rvert_{_2}$}\\
Linearly constrained (fixed box)     &  1316 &  129 & 0.90 &    15 &  1508\\
Linearly constrained (variable box)  &   797 &   45 & 0.94 &    30 &  2263\\
Nonlinearly constrained              & 24055 & 4284 & 0.82 & 25383 & 16909\\
Nonlinearly constrained (*)          & 18038 & 3112 & 0.83 &  3719 & 12635\\
\hline
\end{tabular}
\end{tiny}
\end{center}
where we make use of
(\ref{NumericalResults:Remark:p1}) and in (*) we consider only those examples for which MPBNGC satisfied one of its termination criterions
(cf.~Subsection \ref{section:ExclusionBoxTests:NonLinearlyConstrainedCase}).

For the linearly constrained problems MPBNGC was the fastest of the tested algorithms, followed by BNLC and SolvOpt.
If we consider only those nonlinearly constrained examples for which MPBNGC was able to terminate successfully, MPBNGC was the fastest algorithm again. Considering the competitors, for the nonlinearly constrained problems with $T=1$ the reduced algorithm is 13.3 seconds resp.~11.3 seconds faster than SolvOpt, while for the nonlinearly constrained problems with $T=\lvert y\rvert_{_2}$ SolvOpt is 7.1 seconds resp.~5.4 seconds faster than the reduced algorithm.

Taking
a closer look at $p_1$ yields the observation that at least 85\% of the time is consumed by solving the QP (in the linearly constrained case) resp.~at least 80\% of the time is consumed by solving the QCQP (in the nonlinearly constrained case), which implies that the difference in the percentage between the QP and the QCQP is small in particular (an investigation of the behavior of the solving time $t_1$ for higher dimensional problems can be found in
\citet[\GenaueAngabeHD]{HannesPaperA}).

Therefore, we will concentrate in
Subsection
\ref{section:ExclusionBoxTests:LinearlyConstrainedCaseFixedInterval},
Subsection
\ref{section:ExclusionBoxTests:LinearlyConstrainedCaseVariableInterval} and
Subsection
\ref{section:ExclusionBoxTests:NonLinearlyConstrainedCase}
on the comparison of qualitative aspects between \HauptAlgorithmusKomma MPBNGC and SolvOpt (like, e.g., the cost $c$ of the evaluations),
where before making these detailed comparisons, we give a short overview of them as a reason of clarity of the presentation: In both cases
$T=1$ (solid line) and
$T=\lvert y\rvert_{_2}$ (dashed line),
where we use the two different line types for a better distinction in the following, we tested
128 linearly constrained examples with a fixed box,
117 linearly constrained examples with a variable box and
201 nonlinearly constrained examples,
which yields the following two summary tables consisting of the number of examples for which
\HauptAlgorithmus(BNLC resp.~the reduced algorithm) is better than MPBNGC resp.~SolvOpt (and vice versa) with respect to the cost $c$ of the evaluations
\begin{center}
\begin{tiny}
\begin{tabular}{l|rrrrrrrr|}
(Color code: Light grey)
& \multicolumn{4}{c|}{MPBNGC} & & \multicolumn{3}{|c|}{BNLC/Red Alg}\\
                                 &
\multicolumn{1}{c}{no termi-}     & 
\multicolumn{1}{c}{significantly} & \multicolumn{1}{c}{better} & \multicolumn{1}{c}{a bit} &
\multicolumn{1}{|c|}{nearly} &
\multicolumn{1}{c}{a bit} & \multicolumn{1}{c}{better} & \multicolumn{1}{c|}{significantly}
\\
                                 &
\multicolumn{1}{c}{nation}     & 
\multicolumn{1}{c}{better} & \multicolumn{1}{c}{} & \multicolumn{1}{c}{better} &
\multicolumn{1}{|c|}{equal} &
\multicolumn{1}{c}{better} & \multicolumn{1}{c}{} & \multicolumn{1}{c|}{better}
\\
\hline
                                     &    &  &  &  & \multicolumn{1}{c}{$T=1$} & & &\\
Linearly constrained (fixed box)     &  0 & 2 & 5  & 12 & 106 &  2 & 0 & 1\\
Linearly constrained (variable box)  &  0 & 0 & 0  &  1 & 116 &  0 & 0 & 0\\
Nonlinearly constrained              & 32 & 6 & 28 & 89 &  31 & 10 & 2 & 3\\
\hline
                                     &  &  &  &  & \multicolumn{1}{c}{$T=\lvert y\rvert_{_2}$} & & &\\
Linearly constrained (fixed box)     &  0 & 2 & 5 &  30 &  91 &  0 &  0 & 0\\
Linearly constrained (variable box)  &  0 & 0 & 0 &   5 & 112 &  0 &  0 & 0\\
Nonlinearly constrained              & 43 & 4 & 28 & 59 &  30 & 15 & 14 & 8\\
\hline
\end{tabular}
\end{tiny}
\end{center}
\begin{center}
\begin{tiny}
\begin{tabular}{l|rrrrrrrr|}
(Color code: Black)
& \multicolumn{4}{c|}{SolvOpt} & & \multicolumn{3}{|c|}{BNLC/Red Alg}\\
                                 &
\multicolumn{1}{c}{no termi-}     & 
\multicolumn{1}{c}{significantly} & \multicolumn{1}{c}{better} & \multicolumn{1}{c}{a bit} &
\multicolumn{1}{|c|}{nearly} &
\multicolumn{1}{c}{a bit} & \multicolumn{1}{c}{better} & \multicolumn{1}{c|}{significantly}
\\
                                 &
\multicolumn{1}{c}{nation}     & 
\multicolumn{1}{c}{better} & \multicolumn{1}{c}{} & \multicolumn{1}{c}{better} &
\multicolumn{1}{|c|}{equal} &
\multicolumn{1}{c}{better} & \multicolumn{1}{c}{} & \multicolumn{1}{c|}{better}
\\
\hline
                                     &  &  &  &  & \multicolumn{1}{c}{$T=1$} & & &\\
Linearly constrained (fixed box)     & 0 & 1 &  3 &  0 & 61 & 25 & 13 & 25\\
Linearly constrained (variable box)  & 0 & 0 &  0 &  0 & 48 & 37 & 24 & 8\\
Nonlinearly constrained              & 0 & 0 & 14 & 20 & 21 & 76 & 20 & 50\\
\hline
                                     &  &  &  &  & \multicolumn{1}{c}{$T=\lvert y\rvert_{_2}$} & & &\\
Linearly constrained (fixed box)     & 0 & 1 &  2 &  1 & 32 & 34 & 49 &  9\\
Linearly constrained (variable box)  & 0 & 0 &  0 &  5 & 41 & 32 & 19 & 20\\
Nonlinearly constrained              & 0 & 2 & 24 & 26 & 31 & 61 & 45 & 12\\
\hline
\end{tabular}
\end{tiny}
\end{center}
that are visualized in
Figures \ref{Figure:OverviewTableLinearlyConstrainedFixedBox},
\ref{Figure:OverviewTableLinearlyConstrainedVariableBox}, and
\ref{Figure:OverviewTableNonlinearlyConstrained}
\\
\begin{center}
\captionsetup{type=figure}
\includegraphics[width=8cm]{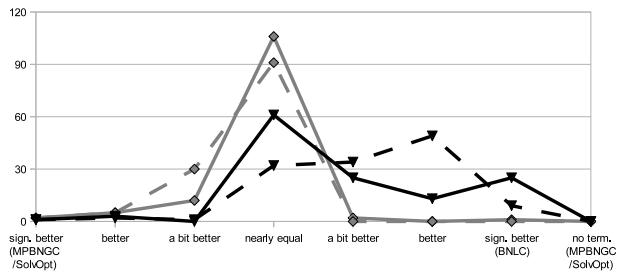}
\captionof{figure}{Linearly constrained with fixed box (summary)}
\label{Figure:OverviewTableLinearlyConstrainedFixedBox}
\end{center}
\begin{center}
\captionsetup{type=figure}
\includegraphics[width=8cm]{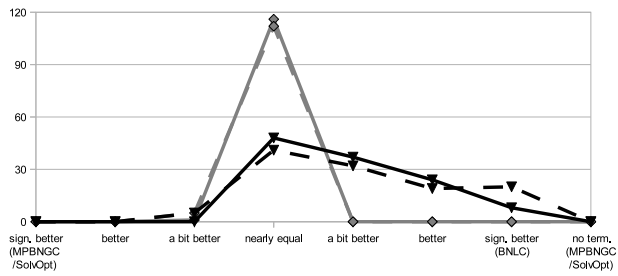}
\captionof{figure}{Linearly constrained with variable box (summary)}
\label{Figure:OverviewTableLinearlyConstrainedVariableBox}
\end{center}
\begin{center}
\captionsetup{type=figure}
\includegraphics[width=8cm]{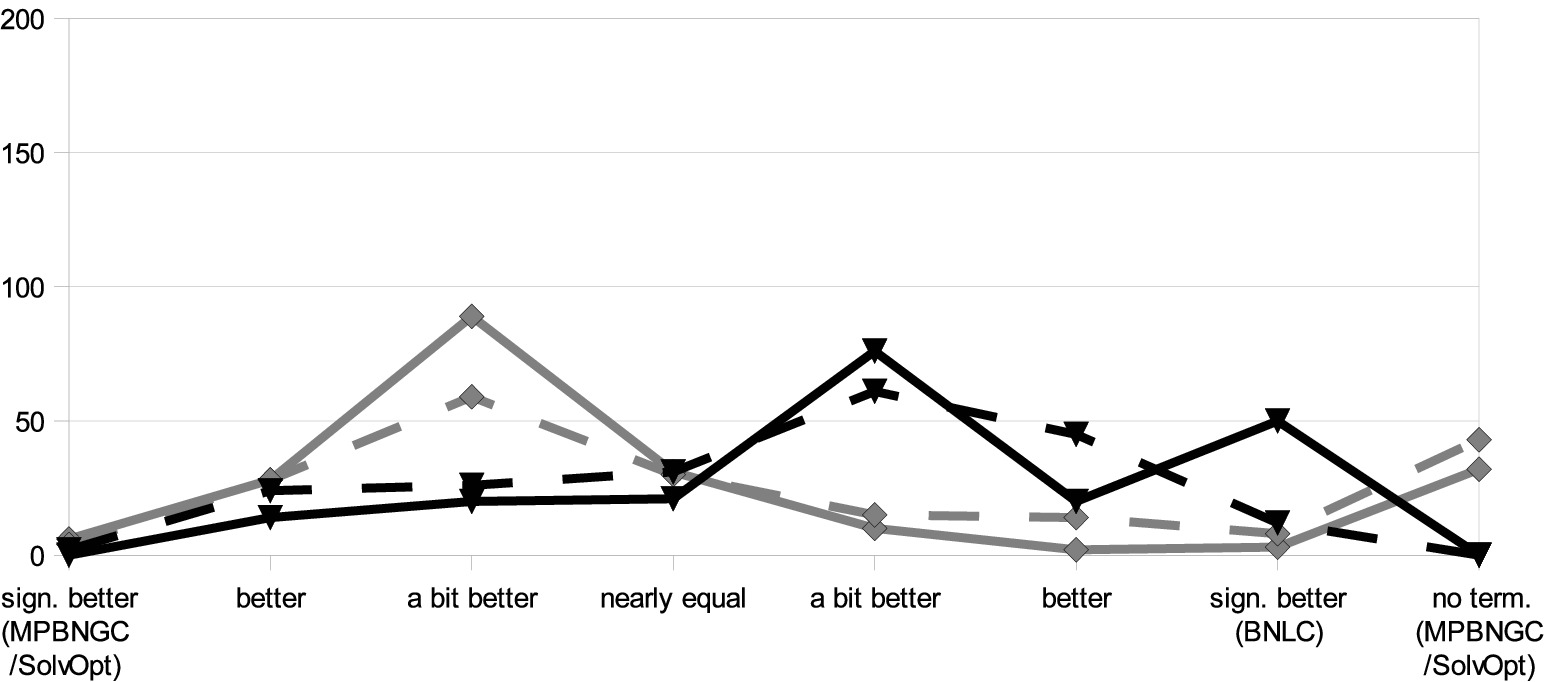}
\captionof{figure}{Nonlinearly constrained (summary)}
\label{Figure:OverviewTableNonlinearlyConstrained}
\end{center}
and that let us draw the following conclusions:

The performance differences between BNLC and MPBNGC can be neglected for the largest part of the linearly constrained examples (with small advantages for MPBNGC in about ten percent of these examples).
For the nonlinearly constrained examples the reduced algorithm is superior to MPBNGC in one quarter of the examples, for forty percent of the examples one of these two solvers has small advantages over the other (in most cases MPBNGC is the slightly more successful one), the performance differences between the two algorithms considered can be completely neglected for fifteen percent of the examples, and for further fifteen percent of the examples MPBNGC beats the reduced algorithm clearly.

For the linearly constrained examples BNLC is superior to SolvOpt in one third of the examples, for one quarter of the examples one of these two solvers has small advantages over the other (in nearly all cases BNLC is the slightly more successful one), the performance differences between the two algorithms considered can be completely neglected for forty percent of the examples, and in only one percent of the examples SolvOpt beats the reduced algorithm clearly.
For the nonlinearly constrained examples the reduced algorithm is superior to SolvOpt in one third of the examples, for 45 percent of the examples one of these two solvers has small advantages over the other (the reduced algorithm is often the slightly more successful one), the performance differences between the considered two algorithms can be completely neglected for ten percent of the examples, and in the remaining ten percent of the examples SolvOpt beats the reduced algorithm clearly.

In contrast to the linearly constrained case, in which all three solvers terminated successfully for all examples, only the reduced algorithm and SolvOpt were able to attain this goal in the nonlinearly constrained case, too.

\subsection{Linearly constrained case (fixed box)}
\label{section:ExclusionBoxTests:LinearlyConstrainedCaseFixedInterval}
We took 310 examples from real CSPs that occur in GloptLab. We observe that
for 79 examples
the starting point is feasible for the CSP and
for 103 examples
the evaluation of the certificate at the starting point identifies the box as infeasible
and hence there remain 128 test problems.

\paragraph{BNLC vs. MPBNGC}
In the case $T=1$ we conclude from Figure \ref{NumericalResults:Figure:LCfixed-RP-c-T1_AddonMPBNGC} that BNLC is
significantly better in 1 example
and
a bit better in 2 examples
in comparison with MPBNGC, while MPBNGC is
significantly better in 2 examples,
better in 5 examples
and
a bit better in 12 examples
in comparison with BNLC. In the 106 remaining examples the costs of BNLC and MPBNGC are practically the same.

In the case $T=\lvert y\rvert_{_2}$ it follows from Figure \ref{NumericalResults:Figure:LCfixed-RP-c-TNormy_AddonMPBNGC} that MPBNGC is
significantly better in 2 examples,
better in 5 examples
and
a bit better in 30 examples
in comparison with BNLC. In the 91 remaining examples the costs of BNLC and MPBNGC are practically the same.

\paragraph{BNLC vs. SolvOpt}
In the case $T=1$ we conclude from Figure \ref{NumericalResults:Figure:LCfixed-RP-c-T1} that BNLC is
significantly better in 25 examples,
better in 13 examples
and
a bit better in 25 examples
in comparison with SolvOpt, while SolvOpt is
significantly better in 1 example
and
better in 3 examples
in comparison with BNLC. In the 61 remaining examples the costs of BNLC and SolvOpt are practically the same.

In the case $T=\lvert y\rvert_{_2}$ it follows from Figure \ref{NumericalResults:Figure:LCfixed-RP-c-TNormy} that BNLC is
significantly better in 9 examples,
better in 49 examples
and
a bit better in 34 examples
in comparison with SolvOpt, while SolvOpt is
significantly better in 1 example,
better in 2 examples
and
a bit better in 1 example
in comparison with BNLC. In the 32 remaining examples the costs of BNLC and SolvOpt are practically the same.

\subsection{Linearly constrained case (variable box)}
\label{section:ExclusionBoxTests:LinearlyConstrainedCaseVariableInterval}
We observe that
for 80 examples
the starting point is feasible for the CSP and
for 113 examples
the evaluation of the certificate at the starting point identifies the boxes as infeasible
and hence there remain 117 test problems of the 310 original examples from GloptLab.

\paragraph{BNLC vs. MPBNGC}
In the case $T=1$ we conclude from Figure \ref{NumericalResults:Figure:LCvariable_RP_c_T1_AddonMPBNGC} that MPBNGC is
a bit better in 1 example
in comparison with BNLC. In the 116 remaining examples the costs of BNLC and MPBNGC are practically the same.

In the case $T=\lvert y\rvert_{_2}$ it follows from Figure \ref{NumericalResults:Figure:LCvariable-RP-c-TNormy_AddonMPBNGC} that MPBNGC is
a bit better in 5 examples
in comparison with BNLC. In the 112 remaining examples the costs of BNLC and MPBNGC are practically the same.

\paragraph{BNLC vs. SolvOpt}
In the case $T=1$ we conclude from Figure \ref{NumericalResults:Figure:LCvariable_RP_c_T1} that BNLC is
significantly better in 8 examples,
better in 24 examples
and
a bit better in 37 examples
in comparison with SolvOpt. In the 48 remaining examples the costs of BNLC and SolvOpt are practically the same.

In the case $T=\lvert y\rvert_{_2}$ it follows from Figure \ref{NumericalResults:Figure:LCvariable-RP-c-TNormy} that BNLC is
significantly better in 20 examples,
better in 19 examples
and
a bit better in 32 examples
in comparison with SolvOpt, while SolvOpt is
a bit better in 5 examples (21, 101, 102, 128, 189)
in comparison with BNLC. In the 41 remaining examples the costs of BNLC and SolvOpt are practically the same.

\subsection{Nonlinearly constrained case}
\label{section:ExclusionBoxTests:NonLinearlyConstrainedCase}
Since we were not able to find a starting point, i.e.~an infeasible sub-box, for 109 examples,
we exclude them from the following tests for which there remain 201 examples of the 310 original examples from GloptLab.

\paragraph{Reduced algorithm vs. MPBNGC}
In the case $T=1$ MPBNGC does not satisfy any of its termination criterions for 32 examples
within the given number of iterations and function evaluations.
For the remaining
169 examples we conclude from Figure \ref{NumericalResults:Figure:NLC_RP_c_T1_AddonMPBNGC} that the reduced algorithm is
significantly better in 3 examples,
better in 2 examples
and
a bit better in 10 examples
in comparison with MPBNGC, while MPBNGC is
significantly better in 6 examples,
better in 28 examples
and
a bit better in 89 examples
in comparison with the reduced algorithm, and in
31
examples the costs of the reduced algorithm and MPBNGC are practically the same.

In the case $T=\lvert y\rvert_{_2}$ MPBNGC does not satisfy any of its termination criterions for 43 examples
within the given number of iterations and function evaluations. For the remaining
158
examples it follows from Figure \ref{NumericalResults:Figure:NLC_RP_c_TNormy_AddonMPBNGC} that the reduced algorithm is
significantly better in 8 examples,
better in 14 examples
and
a bit better in 15 examples
in comparison with MPBNGC, while MPBNGC is
significantly better in 4 examples,
better in 28 examples
and
a bit better in 59 examples
in comparison with the reduced algorithm, and in
30
examples the costs of the reduced algorithm and MPBNGC are practically the same.

\paragraph{Reduced algorithm vs. SolvOpt}
In the case $T=1$ we conclude from Figure \ref{NumericalResults:Figure:NLC_RP_c_T1}  that the reduced algorithm is
significantly better in 50 examples,
better in 20 examples
and
a bit better in 76 examples
in comparison with SolvOpt, while SolvOpt is
better in 14 examples
and
a bit better in 20 examples
in comparison with the reduced algorithm. In the 21 remaining examples the costs of the reduced algorithm and SolvOpt are practically the same.

In the case $T=\lvert y\rvert_{_2}$ it follows from Figure \ref{NumericalResults:Figure:NLC_RP_c_TNormy} that the reduced algorithm is
significantly better in 12 examples,
better in 45 examples
and
a bit better in 61 examples
in comparison with SolvOpt, while SolvOpt is
significantly better in 2 examples,
better in 24 examples
and
a bit better in 26 examples
in comparison with the reduced algorithm. In the 31 remaining examples the costs of the reduced algorithm and SolvOpt are practically the same.

\section{Conclusion}
In this paper we presented a nonsmooth function that can be used as a certificate of infeasibility that allows the identification of exclusion boxes during the solution process of a CSP by techniques from nonsmooth optimization: While we can find an exclusion box by solving a linearly constrained nonsmooth optimization problem, the enlargement of an exclusion box can be achieved by solving a nonlinearly constrained nonsmooth optimization problem. Furthermore, we discussed important properties of the certificate as the reduction of scalability and we suggested a method to obtain a good starting point for the nonsmooth optimization problems.





\ifthenelse{\boolean{AppendixOutsourcing}}
{
\newpage
}

\begin{appendix}
\ifthenelse{\boolean{AppendixOutsourcing}}
{
\notocsection{Figures}
}
{
\section{Figures}
}
\begin{center}
\captionsetup{type=figure}
\includegraphics[width=14cm]{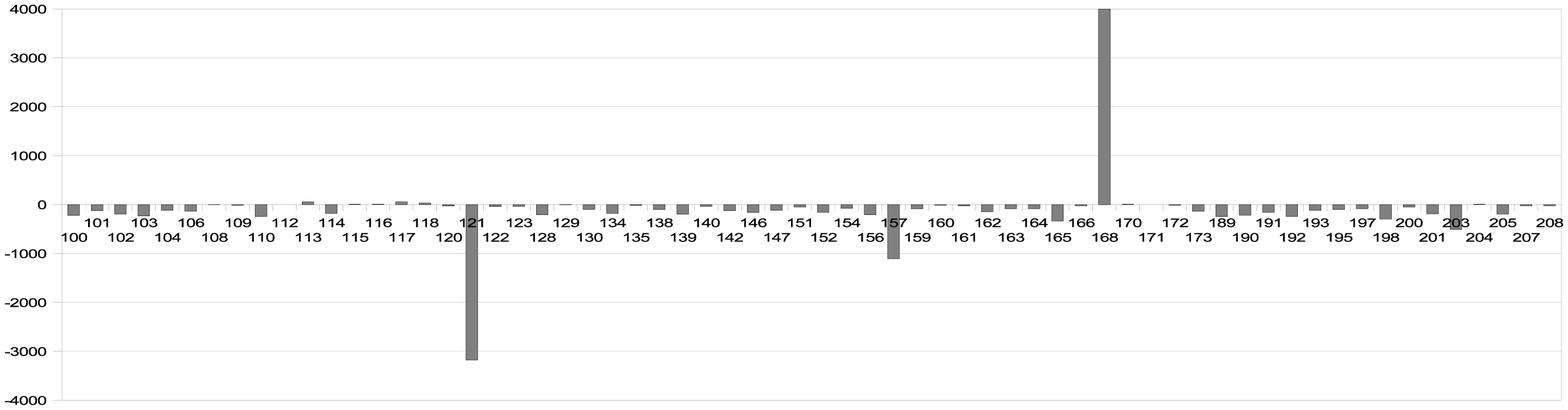}
\includegraphics[width=14cm]{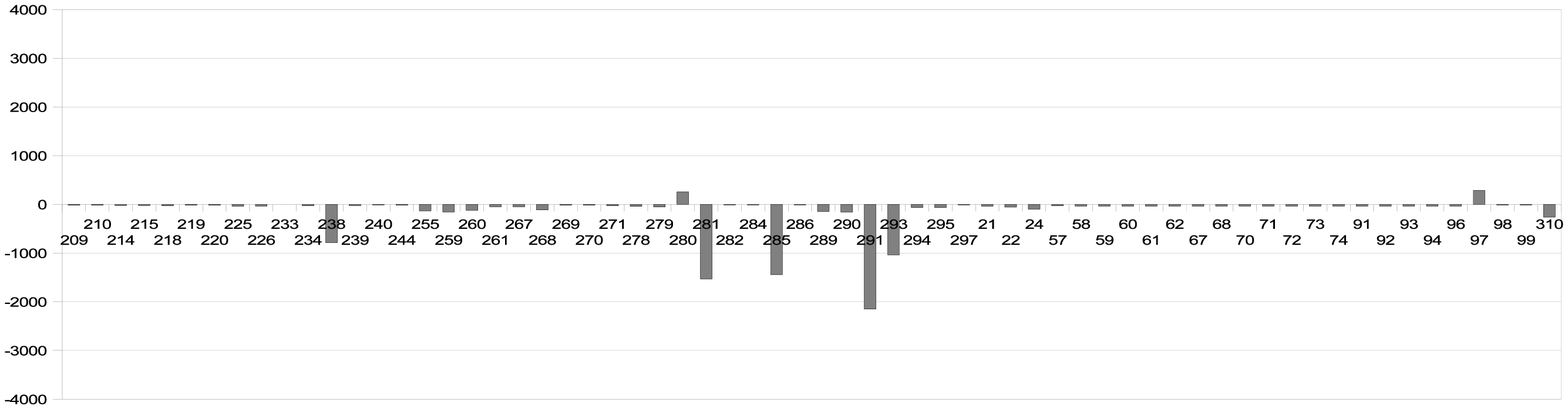}
\captionof{figure}{Linearly constrained (fixed box) --- \rp(c) for BNLC \& MPBNGC ($T=1$)}
\label{NumericalResults:Figure:LCfixed-RP-c-T1_AddonMPBNGC}
\end{center}
\begin{center}
\captionsetup{type=figure}
\includegraphics[width=14cm]{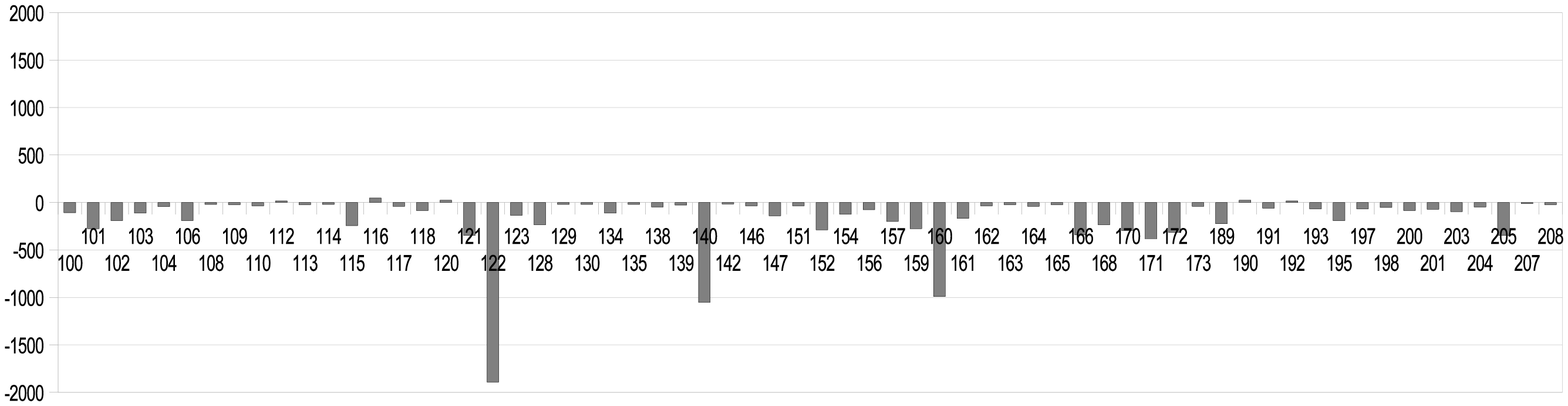}
\includegraphics[width=14cm]{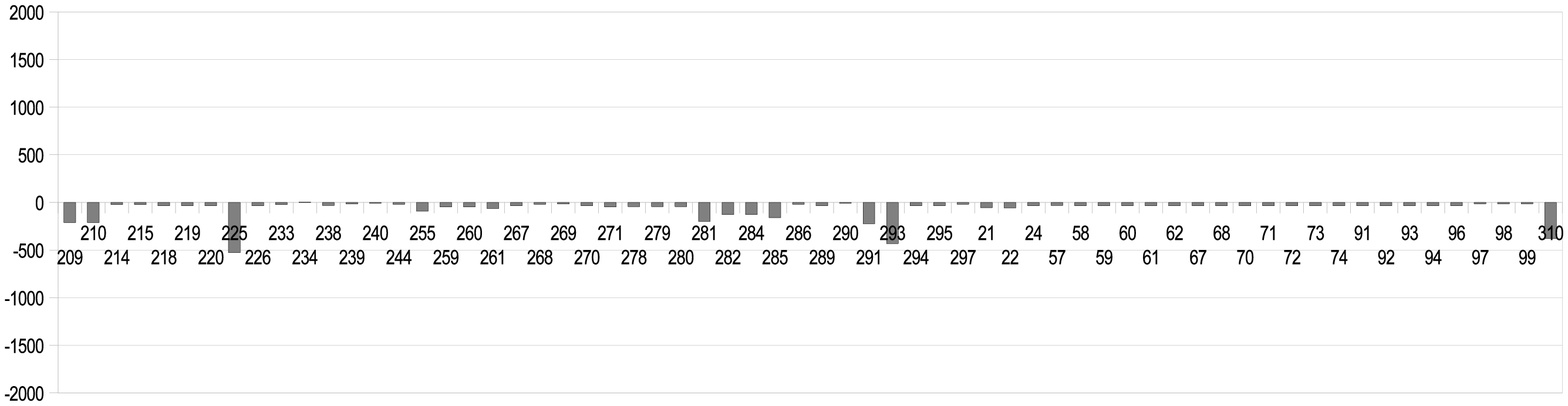}
\captionof{figure}{Linearly constrained (fixed box) --- \rp(c) for BNLC \& MPBNGC ($T=\lvert y\rvert_{_2}$)}
\label{NumericalResults:Figure:LCfixed-RP-c-TNormy_AddonMPBNGC}
\end{center}
\begin{center}
\captionsetup{type=figure}
\includegraphics[width=14cm]{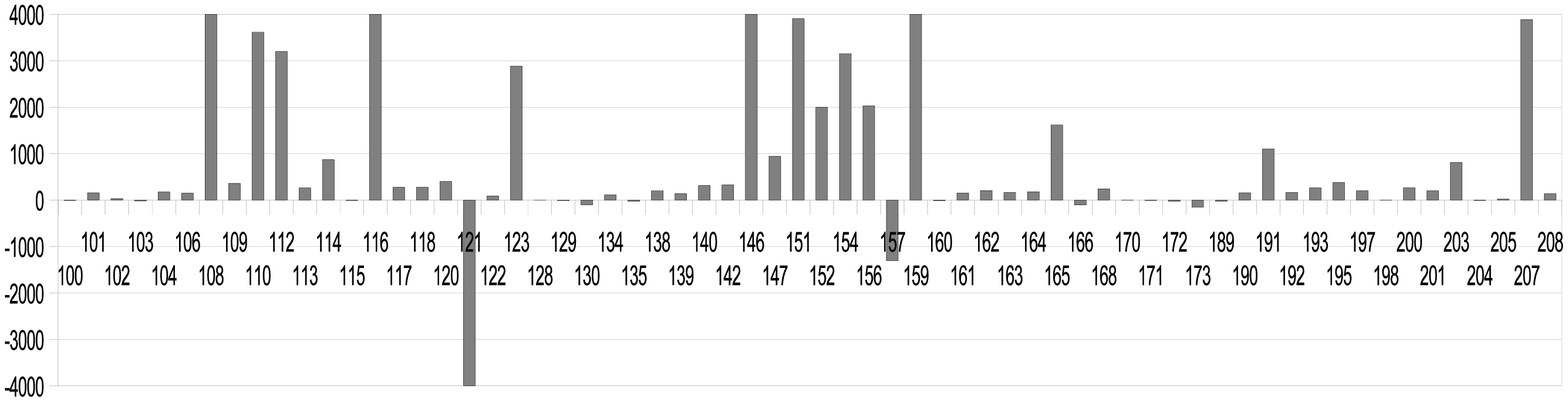}
\includegraphics[width=14cm]{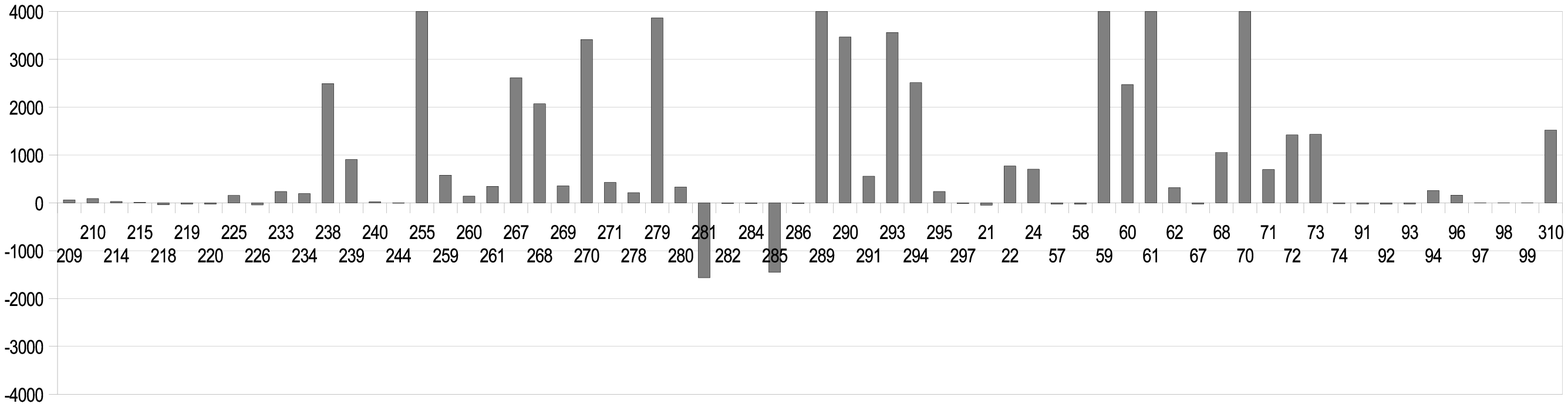}
\captionof{figure}{Linearly constrained (fixed box) --- \rp(c) for BNLC \& SolvOpt ($T=1$)}
\label{NumericalResults:Figure:LCfixed-RP-c-T1}
\end{center}
\begin{center}
\captionsetup{type=figure}
\includegraphics[width=14cm]{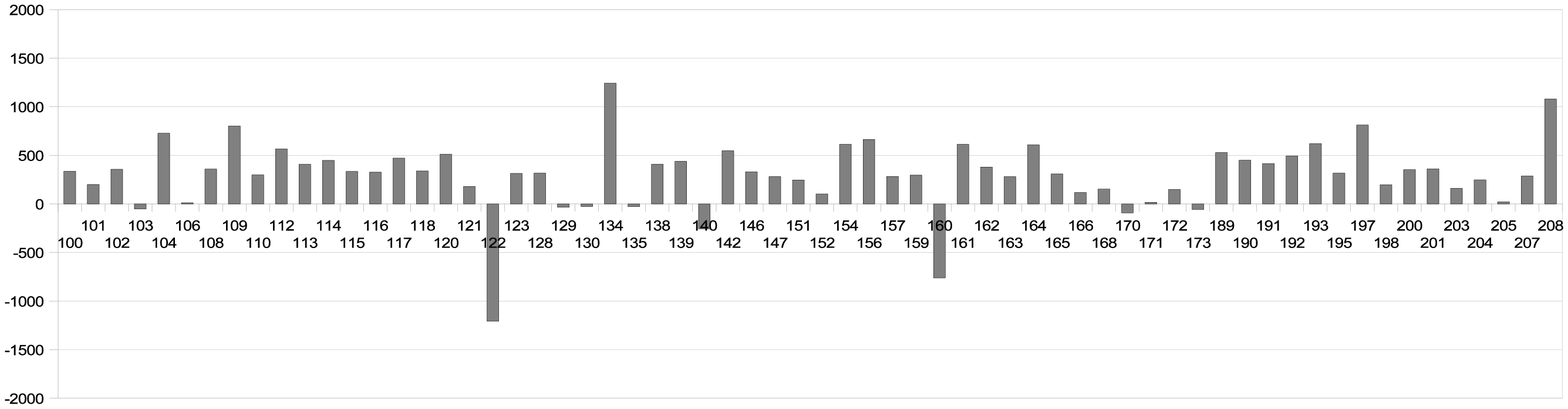}
\includegraphics[width=14cm]{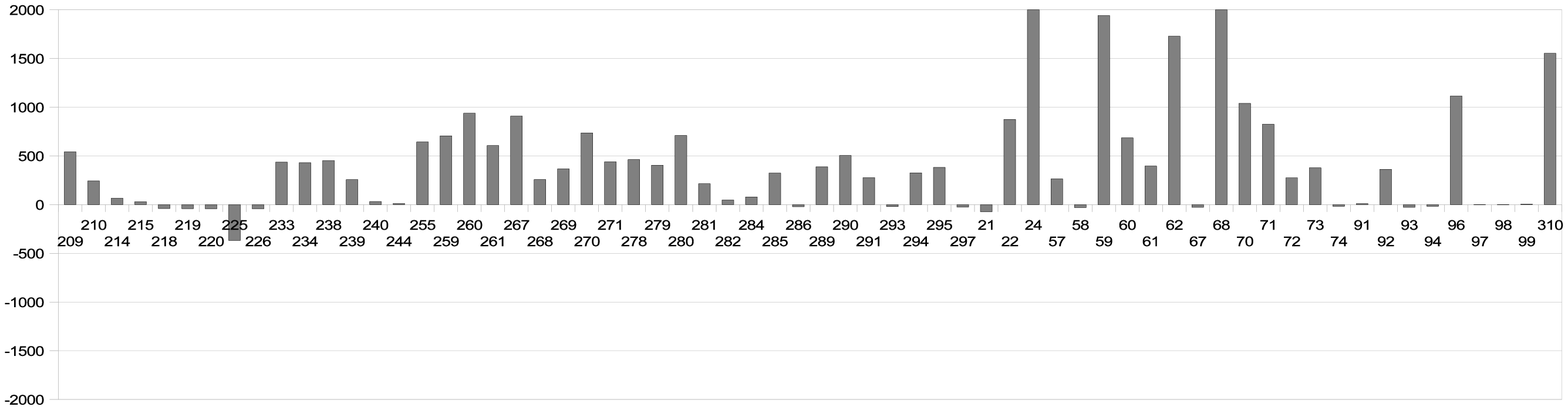}
\captionof{figure}{Linearly constrained (fixed box) --- \rp(c) for BNLC \& SolvOpt ($T=\lvert y\rvert_{_2}$)}
\label{NumericalResults:Figure:LCfixed-RP-c-TNormy}
\end{center}
\begin{center}
\captionsetup{type=figure}
\includegraphics[width=14cm]{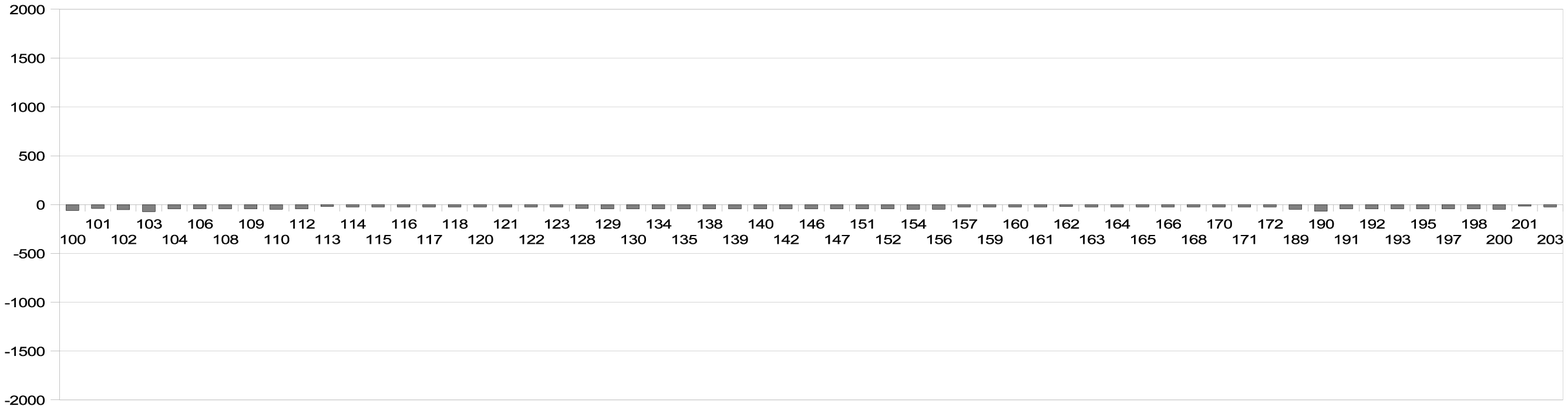}
\includegraphics[width=14cm]{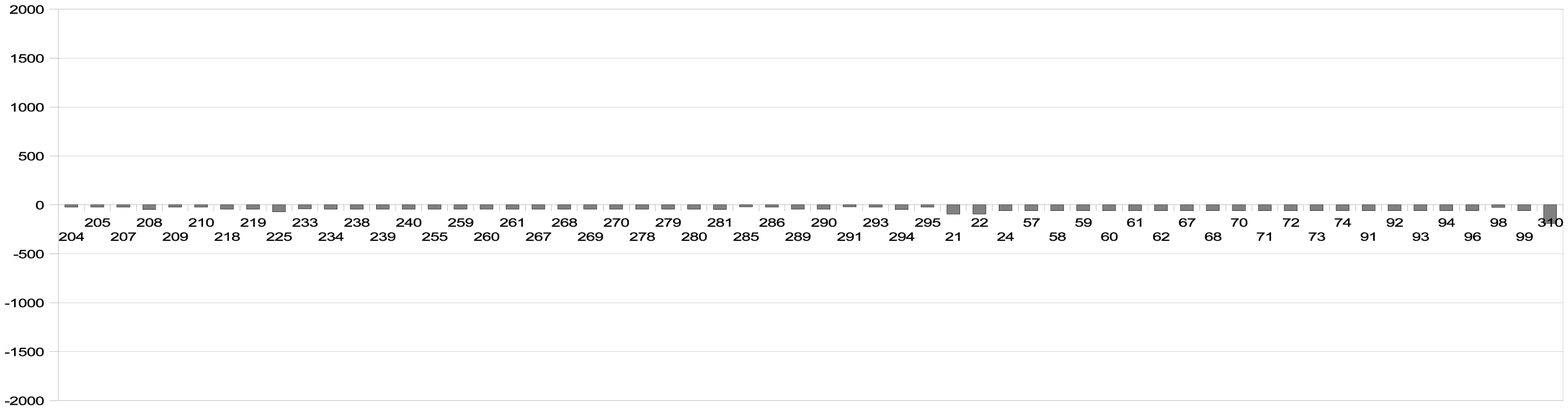}
\captionof{figure}{Linearly constrained (variable box) --- \rp(c) for BNLC \& MPBNGC ($T=1$)}
\label{NumericalResults:Figure:LCvariable_RP_c_T1_AddonMPBNGC}
\end{center}
\begin{center}
\captionsetup{type=figure}
\includegraphics[width=14cm]{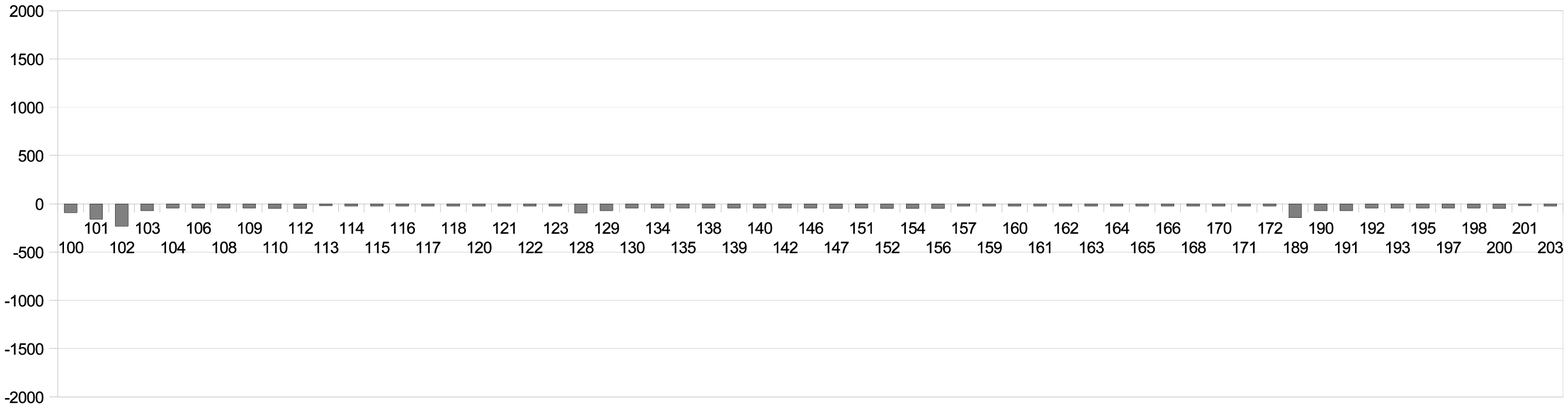}
\includegraphics[width=14cm]{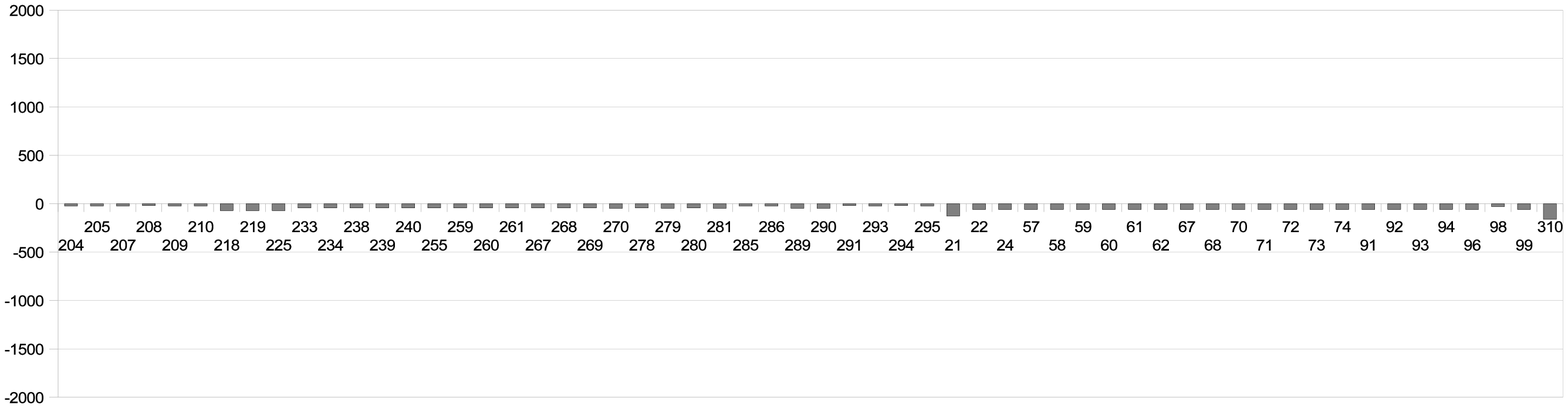}
\captionof{figure}{Linearly constrained (variable box) --- \rp(c) for BNLC \& MPBNC ($T=\lvert y\rvert_{_2}$)}
\label{NumericalResults:Figure:LCvariable-RP-c-TNormy_AddonMPBNGC}
\end{center}
\begin{center}
\captionsetup{type=figure}
\includegraphics[width=14cm]{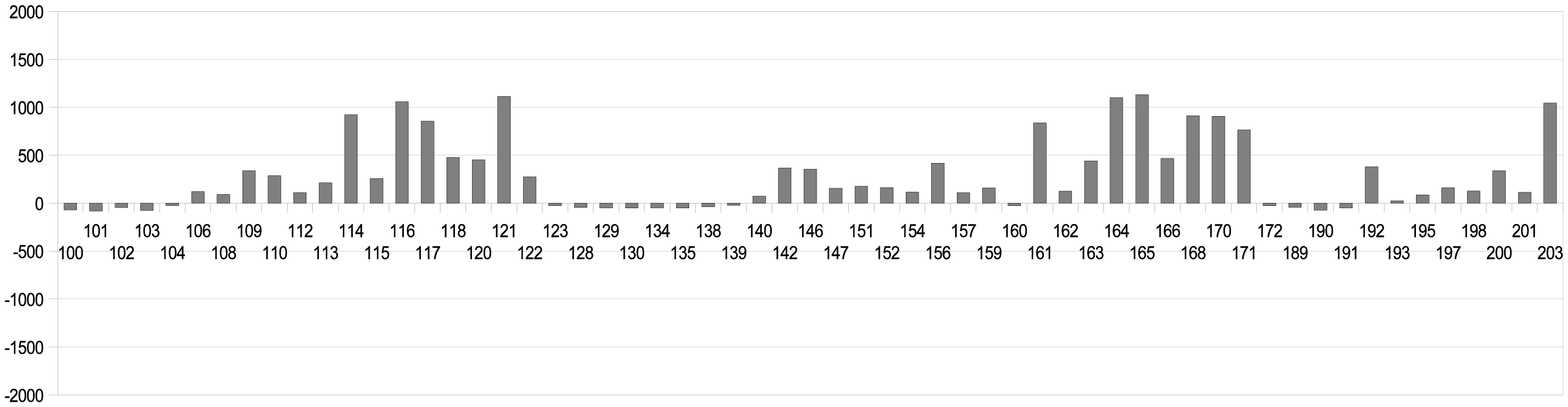}
\includegraphics[width=14cm]{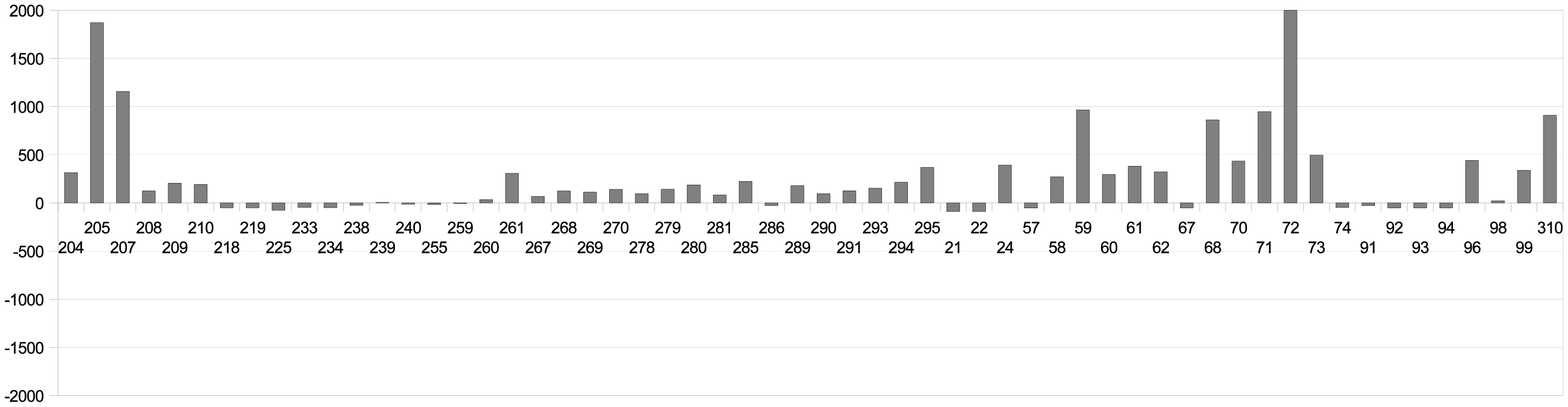}
\captionof{figure}{Linearly constrained (variable box) --- \rp(c) for BNLC \& SolvOpt ($T=1$)}
\label{NumericalResults:Figure:LCvariable_RP_c_T1}
\end{center}
\begin{center}
\captionsetup{type=figure}
\includegraphics[width=14cm]{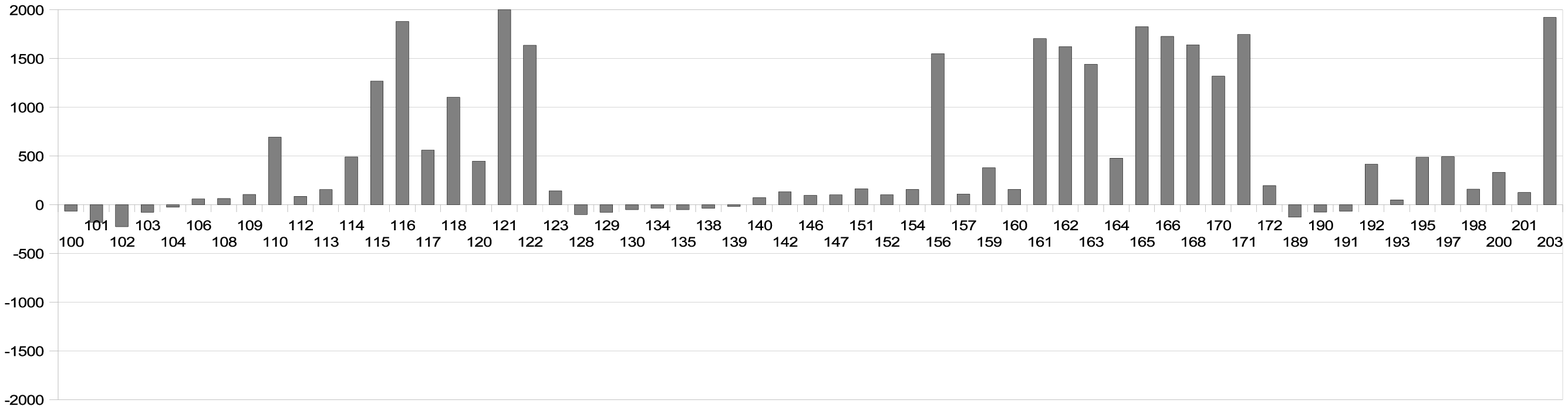}
\includegraphics[width=14cm]{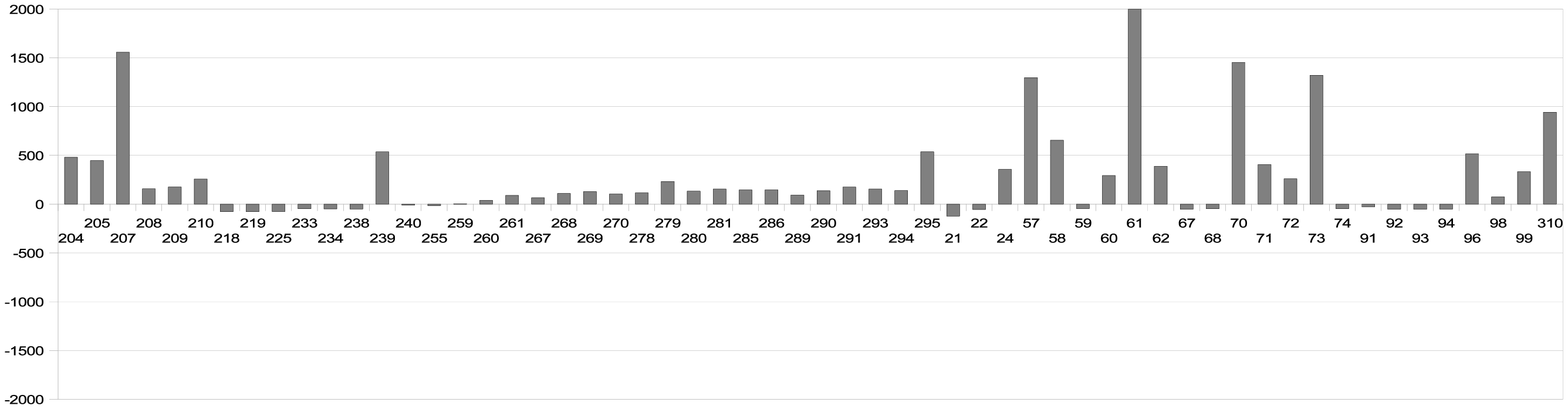}
\captionof{figure}{Linearly constrained (variable box) --- \rp(c) for BNLC \& SolvOpt ($T=\lvert y\rvert_{_2}$)}
\label{NumericalResults:Figure:LCvariable-RP-c-TNormy}
\end{center}
\begin{center}
\captionsetup{type=figure}
\includegraphics[width=14cm]{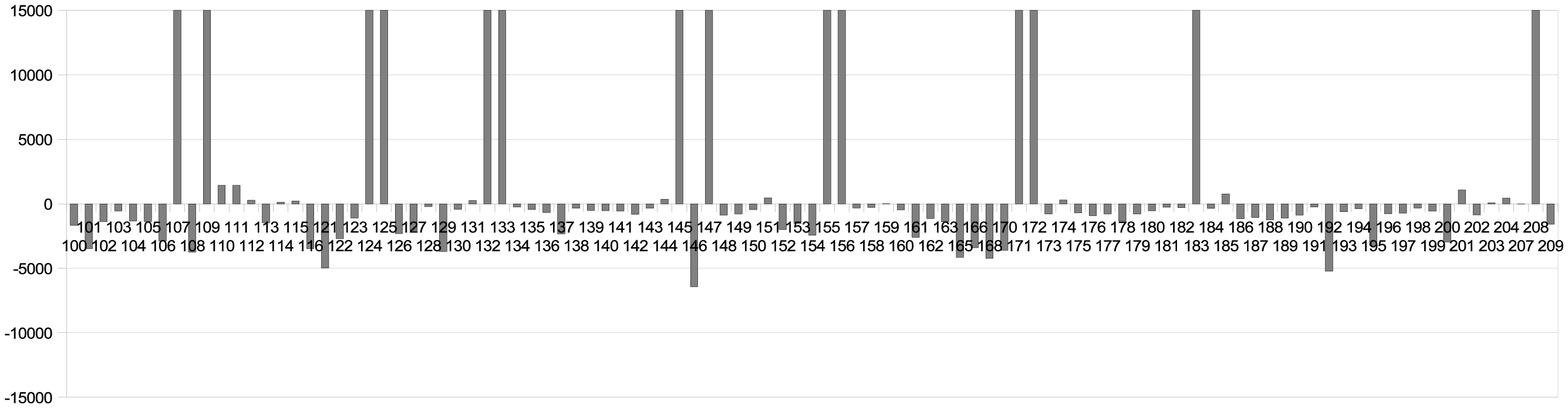}
\includegraphics[width=14cm]{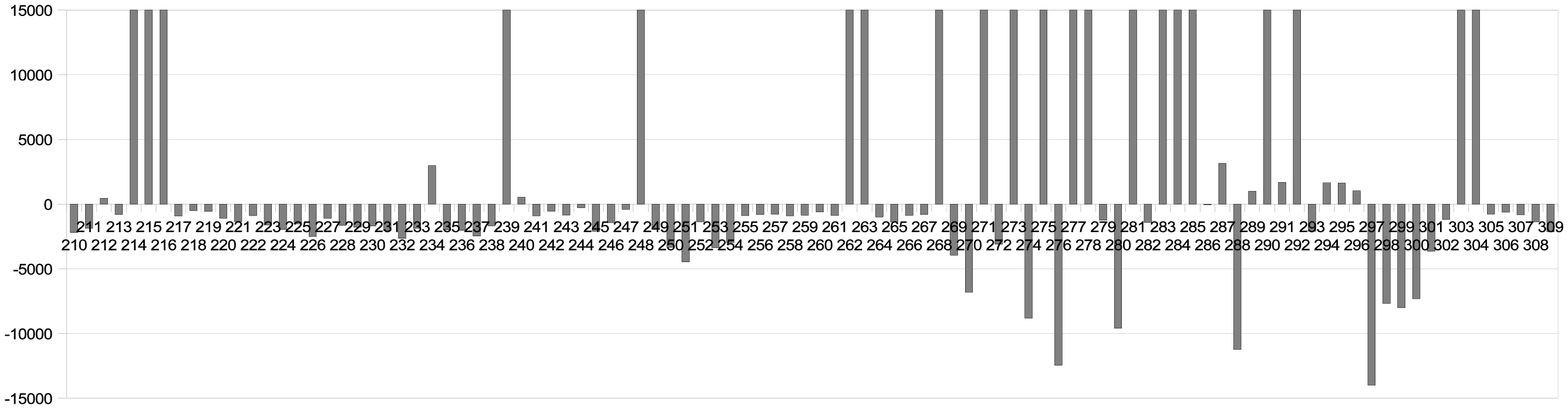}
\captionof{figure}{Nonlinearly constrained --- \rp(c) for Red Alg \& MPBNGC ($T=1$)}
\label{NumericalResults:Figure:NLC_RP_c_T1_AddonMPBNGC}
\end{center}
\begin{center}
\captionsetup{type=figure}
\includegraphics[width=14cm]{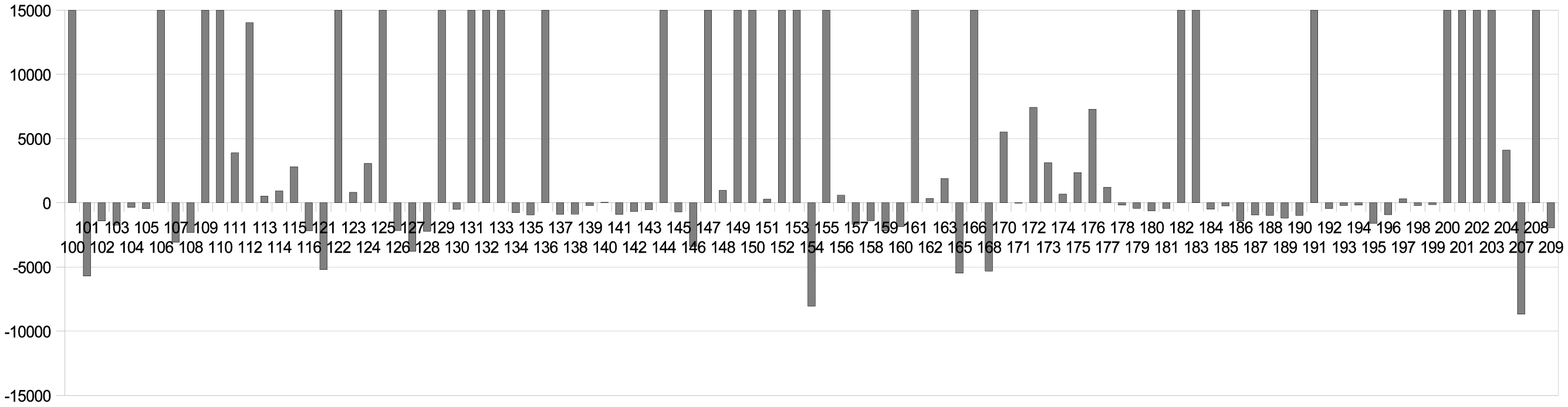}
\includegraphics[width=14cm]{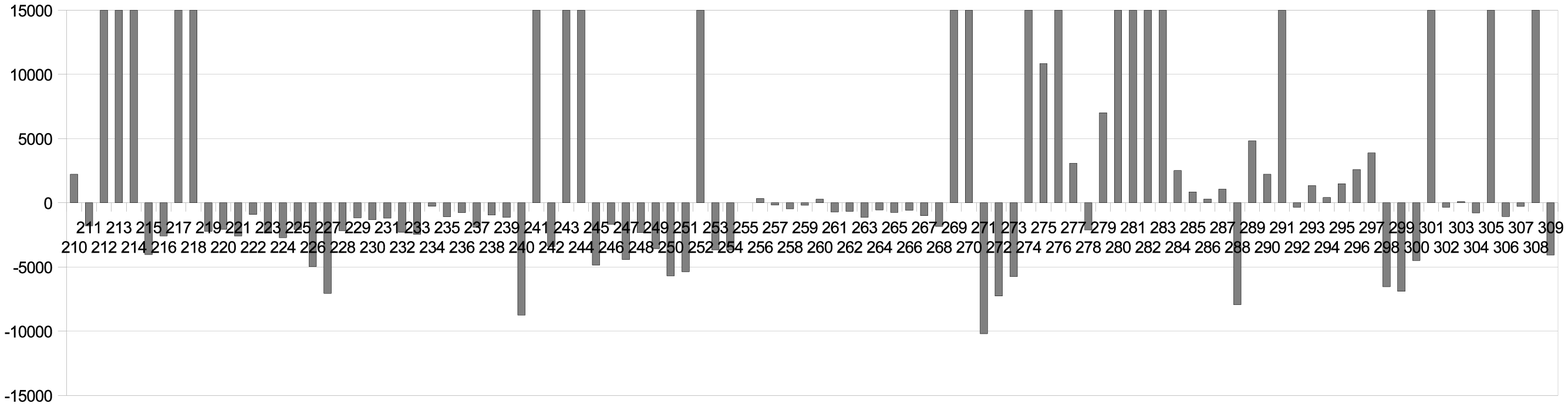}
\captionof{figure}{Nonlinearly constrained --- \rp(c) for Red Alg \& MPBNGC ($T=\lvert y\rvert_{_2}$)}
\label{NumericalResults:Figure:NLC_RP_c_TNormy_AddonMPBNGC}
\end{center}
\begin{center}
\captionsetup{type=figure}
\includegraphics[width=14cm]{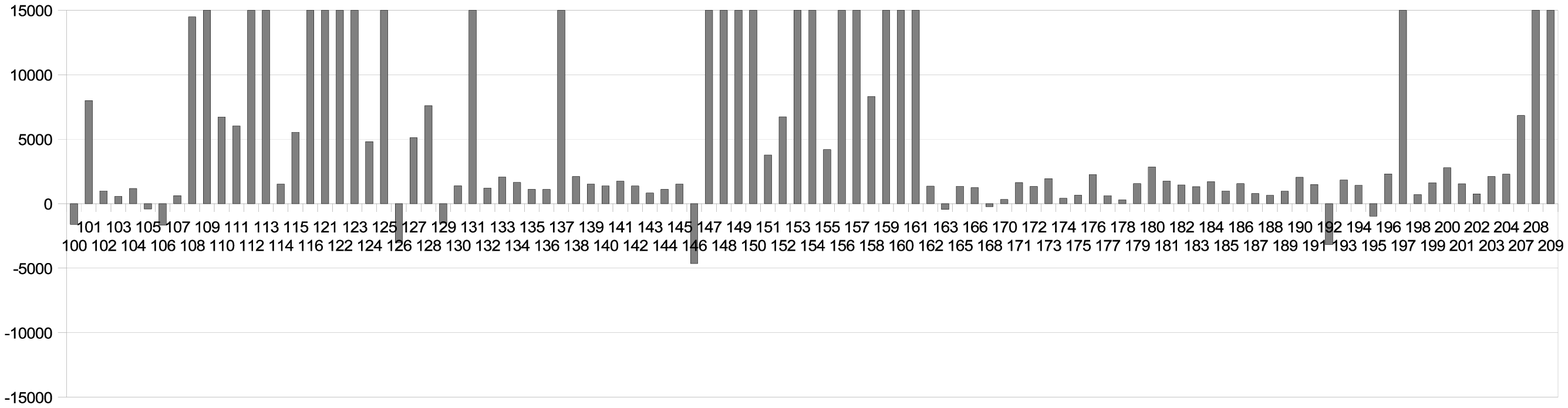}
\includegraphics[width=14cm]{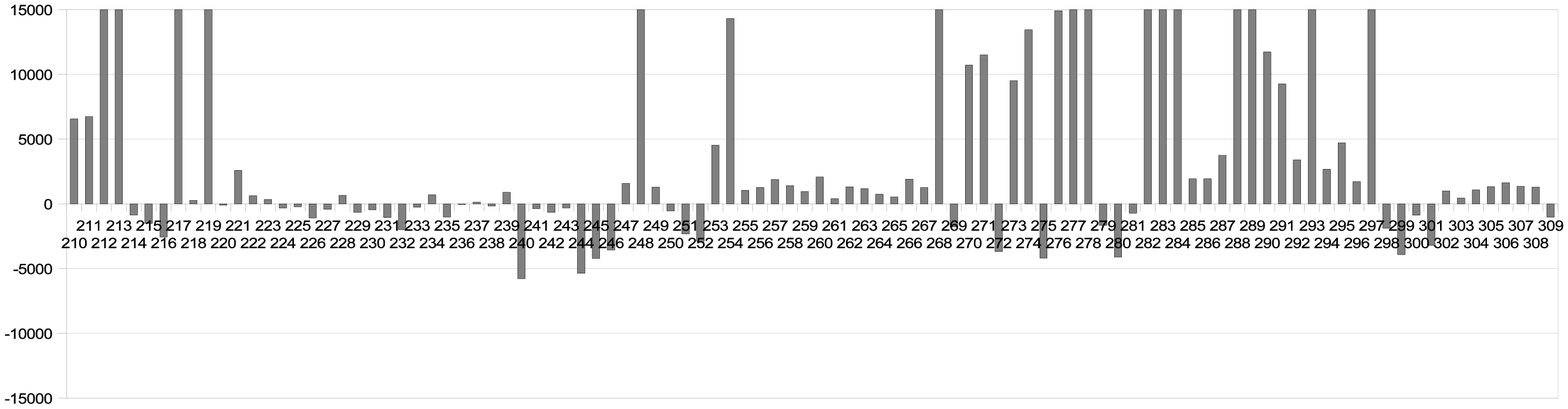}
\captionof{figure}{Nonlinearly constrained --- \rp(c) for Red Alg \& SolvOpt ($T=1$)}
\label{NumericalResults:Figure:NLC_RP_c_T1}
\end{center}
\begin{center}
\captionsetup{type=figure}
\includegraphics[width=14cm]{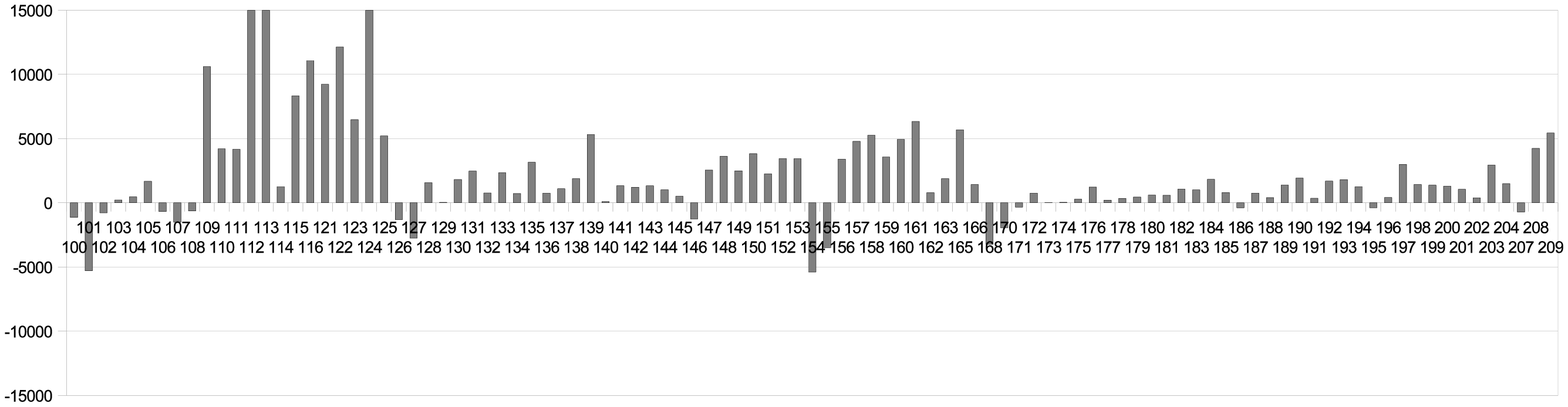}
\includegraphics[width=14cm]{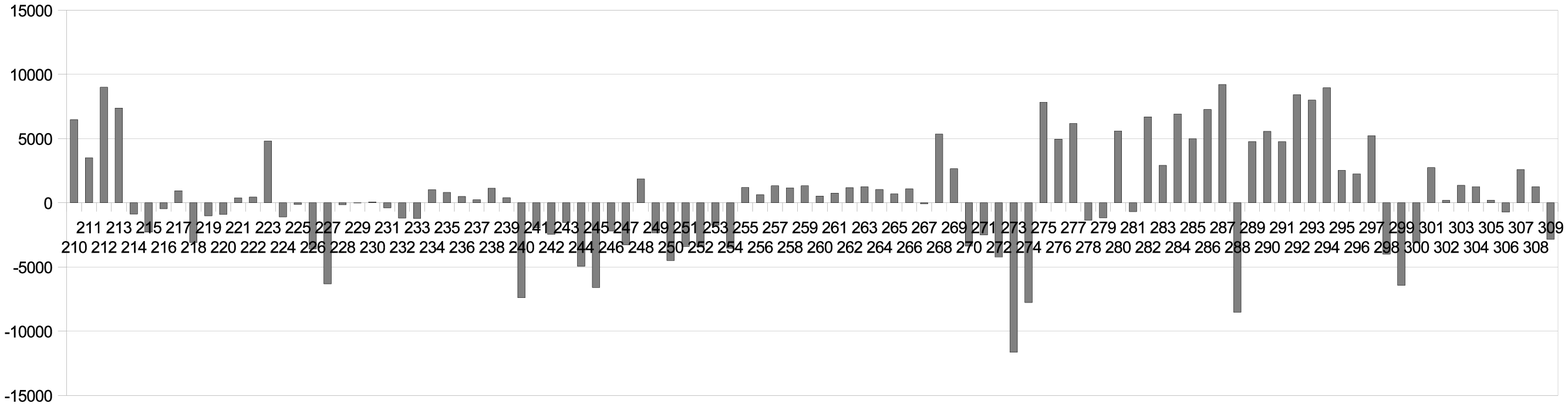}
\captionof{figure}{Nonlinearly constrained --- \rp(c) for Red Alg \& SolvOpt ($T=\lvert y\rvert_{_2}$)}
\label{NumericalResults:Figure:NLC_RP_c_TNormy}
\end{center}
\end{appendix}


\end{document}